\documentclass[twoside,12pt]{article}

\usepackage{amsmath}
\usepackage{amssymb}
\usepackage{graphicx}
\usepackage{theorem}
\usepackage{pifont}
\usepackage{color}
\usepackage{enumitem}
\usepackage{hyperref}
\usepackage{todonotes}
\usepackage{siunitx}
\usepackage{cancel}
\usepackage[normalem]{ulem}
\usepackage{tikz-cd}
\usepackage{pgfkeys}
\usepackage{quiver}
\usepackage[left=1.2in,top=1.2in,right=1.2in,bottom=1.2in]{geometry} 

\usepackage{xcolor} 

\usepackage{amsmath}
\usepackage{amssymb}
\usepackage{graphicx}
\usepackage{eucal}
\usepackage{mathrsfs}
\usepackage{pifont}
\usepackage{color}
\usepackage{cjhebrew}

\usepackage[all]{xy}
\usepackage{tikz} 
\usepackage{tkz-euclide}
\usepackage{pgfplots}

\newcommand{\btkz}{\begin{tikzpicture}}
\newcommand{\etkz}{\end{tikzpicture}}

\newcommand{\brk}[1]{\left(#1\right)}          
\newcommand{\Brk}[1]{\left[#1\right]}          
\newcommand{\BRK}[1]{\left\{#1\right\}}        

\newcommand{\Norm}[1]{\left\| #1 \right\|}     

\newcommand{\Cases}[1]{\begin{cases} #1 \end{cases}}

\newcommand{\deriv}[2]{\frac{d#1}{d#2}}

\newcommand{\secref}[1]{Section~\ref{#1}}

\newcommand{\thmref}[1]{Theorem~\ref{#1}}

\newcommand{\propref}[1]{Proposition~\ref{#1}}
\newcommand{\lemref}[1]{Lemma~\ref{#1}}
\newcommand{\corrref}[1]{Corollary~\ref{#1}}

\newcommand{\beq}{\begin{equation}}
\newcommand{\eeq}{\end{equation}}
\newcommand{\bsplit}{\begin{split}}
\newcommand{\esplit}{\end{split}}
\newcommand{\baligned}{\begin{aligned}}
\newcommand{\ealigned}{\end{aligned}}

\providecommand{\e}{\varepsilon}
\providecommand{\half}{\frac{1}{2}}

\providecommand{\R}{\bbR}
\newcommand{\smallhalf}{\tfrac{1}{2}}

\newcommand{\textand}{\quad\text{ and }\quad}
\newcommand{\Textand}{\qquad\text{ and }\qquad}

\providecommand{\vp}{\varphi}

\newcommand{\End}{{\operatorname{End}}}

\newcommand{\id}{{\operatorname{Id}}}
\newcommand{\image}{{\operatorname{Image}}}

\newcommand{\trace}{{\operatorname{Tr}}}

\newcommand{\calA}{{\mathcal A}}

\newcommand{\calC}{{\mathcal C}}
\newcommand{\calD}{{\mathcal D}}
\newcommand{\calE}{{\mathcal E}}

\newcommand{\calL}{{\mathcal L}}
\newcommand{\calM}{{\mathcal M}}

\newcommand{\calR}{{\mathcal R}}
\newcommand{\calS}{{\mathcal S}}

\newcommand{\frakB}{\mathfrak{B}}

\newcommand{\frakF}{\mathfrak{F}}
\newcommand{\frakG}{\mathfrak{G}}

\newcommand{\frakT}{\mathfrak{T}}

\newcommand{\frakX}{\mathfrak{X}}

\newcommand{\frake}{\mathfrak{e}}

\newcommand{\frakg}{\mathfrak{g}}
\newcommand{\frakh}{\mathfrak{h}}

\newcommand{\frakn}{\mathfrak{n}}

\newcommand{\frakt}{\mathfrak{t}}

\newcommand{\bbK}{{\mathbb K}}

\newcommand{\bbP}{{\mathbb P}}

\newcommand{\bbR}{{\mathbb R}}

\newtheorem{theorem}{Theorem}[section]
\newtheorem{lemma}[theorem]{Lemma}
\newtheorem{proposition}[theorem]{Proposition}

\newtheorem{corollary}[theorem]{Corollary}

\newenvironment{proof}{{\flushleft \emph{Proof}:}}{\hfill\ding{110}}

\newcommand{\M}{\calM}
\newcommand{\dM}{\partial\M}
\newcommand{\VF}{\frakX}

\newcommand{\bra}{\langle}
\newcommand{\ket}{\rangle}
\newcommand{\D}{\calD}

\newcommand{\ortframe}{\lbrace E_j\rbrace_{j=1}^{d}}

\newcommand{\g}{\frakg}
\newcommand{\G}{\frakG}

\newcommand{\gEps}{{\g_\e}}
\newcommand{\h}{\frakh}
\newcommand{\euc}{\frake}
\renewcommand{\S}{\calS}

\newcommand{\Ptt}{\bbP_\e^{\frakt\frakt}}
\newcommand{\Ptn}{\bbP_\e^{\frakt\frakn}}

\newcommand{\Pnn}{\bbP_\e^{\frakn\frakn}}

\newcommand{\fTN}{f^*T\bar{\M}}
\newcommand{\fh}{f^*\bar{\g}}
\newcommand{\nabfh}{\nabla^{f^*\bar{\g}}}

\newcommand{\delfh}{\delta^{\nabfh}}
\newcommand{\Pnfh}{\bbP^{\frakn}_{\fh}}

\newcommand{\gD}{{\g_0}}

\newcommand{\PnD}{\bbP^{\frakn}}
\newcommand{\PttD}{\bbP^{\frakt\frakt}}
\newcommand{\PtnD}{\bbP^{\frakt\frakn}}
\newcommand{\PntD}{\bbP^{\frakn\frakt}}
\newcommand{\PnnD}{\bbP^{\frakn\frakn}}

\newcommand{\nabg}{\nabla^\g}
\newcommand{\nabgEps}{\nabla^{\gEps}}
\newcommand{\nabgD}{\nabla^{\gD}}

\newcommand{\dr}{\partial_r}

\newcommand{\starG}{\star_\g}

\newcommand{\TT}{{\mathrm{TT}}}

\newcommand{\NN}{{\mathrm{NN}}}

\newcommand{\Hg}{H_\g}
\newcommand{\bHg}{\mathbf{H}_\g}
\newcommand{\Fg}{F_\g}
\newcommand{\bFg}{\mathbf{F}_\g}
\newcommand{\Bg}{\mathbf{B}_\g}

\renewcommand{\div}{\operatorname{Div}}
\newcommand{\curl}{\operatorname{Curl}}
\renewcommand{\trace}{\operatorname{tr}}
\renewcommand{\image}{\operatorname{Im}}

\newcommand{\Rm}{\operatorname{Rm}}

\newcommand{\Hess}{\operatorname{Hess}}

\newcommand{\EE}{\mathcal{EE}}
\newcommand{\EC}{\mathcal{EC}}
\newcommand{\CE}{\mathcal{CE}}
\newcommand{\CC}{\mathcal{CC}}
\newcommand{\BH}{\mathcal{BH}}

\newcommand{\dg}{d^{\nabg}}
\newcommand{\deltag}{\delta^{\nabg}}
\newcommand{\dgEps}{d^{\nabgEps}}

\newcommand{\dgD}{d^{\nabgD}}
\newcommand{\deltagD}{\delta^{\nabgD}}

\newcommand{\dgV}{d^{\nabg}_V}
\newcommand{\deltagV}{\delta^{\nabg}_V}
\newcommand{\dgEpsV}{d^{\nabgEps}_V}

\newcommand{\dgDV}{d^{\nabgD}_V}
\newcommand{\deltagDV}{\delta^{\nabgD}_V}

\newcommand{\EwR}[1]{\stackrel{\eqref{#1}}{=}}

\newcommand{\VectorFormsM}{\Omega^{*,*}(\M)}

\newcommand{\VectorFormsKM}{\Omega^{k,m}(\M)}
\newcommand{\derivAtZero}{\left.\deriv{}{t}\right|_0}

\newcommand{\Volume}{\text{Vol}}
\newcommand{\VolumeG}{d\Volume_\g}

\newcommand{\VolumeD}{d\Volume_{\gD}}

\makeatletter
\newcommand*\owedge{\mathpalette\@owedge\relax}
\newcommand*\@owedge[1]{%
  \mathbin{%
    \ooalign{%
      $#1\m@th\bigcirc$\cr
      \hidewidth$#1\m@th\wedge$\hidewidth\cr
    }%
  }%
}
\makeatother

\numberwithin{equation}{section}

\begin{document}

\title{
On Saint-Venant compatibility and stress potentials in manifolds with boundary and constant sectional curvature
}

\author{
Raz Kupferman and
Roee Leder  \footnote{
raz@math.huji.ac.il,
roee.leder@mail.huji.ac.il
}
\\
\\
Institute of Mathematics \\
The Hebrew University \\
Jerusalem 9190401 Israel
}
\maketitle

\begin{abstract}
We address three related problems in the theory of elasticity, formulated in the framework of double forms: the Saint-Venant compatibility condition, the existence and uniqueness of solutions for equations arising in incompatible elasticity, and the existence of stress potentials. The scope of this work is for manifolds with boundary of arbitrary dimension, having constant sectional curvature. 
The central analytical machinery is the regular ellipticity of a boundary-value problem for a bilaplacian operator, and its consequences, which were developed in \cite{KL21a}. One of the novelties of this work is that stress potentials can be used in non-Euclidean geometries, and that the gauge freedom can be exploited to obtain a generalization for the biharmonic equation for the stress potential in dimensions greater than two. 

\end{abstract}

\section{Introduction}
\label{sec:intro}

This article is concerned with generalizations of three classical problems arising in the theory of elasticity. Let $\D\subset\R^d$ be an open, bounded domain having a smooth boundary. Then:
\begin{enumerate}[itemsep=0pt,label=(\alph*)]
\item {\bfseries The Saint-Venant problem}: Given a symmetric $(2,0)$-tensor $\sigma:\D\to\R^{d\times d}$, what are necessary and sufficient conditions for it to be the symmetric gradient of a vector field $u:\D\to\R^d$,
\[
\sigma_{ij} = \partial_i u_j + \partial_j u_i. 
\]
\item {\bfseries Linearized stress equations}: 
Under what conditions there exists a solution to the linear boundary-value problem:
\[
\div\sigma=0
\qquad
\curl\curl\sigma = \calR
\textand
\sigma\cdot\frakn = 0,
\]
where
\[
\begin{gathered}
(\div\sigma)_i = \sum_{j=1}^d \partial_j\sigma_{ij}  \\
(\curl\curl\sigma)_{ijkl} = \partial_{ik}\sigma_{jl} - \partial_{jk}\sigma_{il}  - \partial_{il}\sigma_{jk}  + \partial_{jl}\sigma_{ik}  \\
(\sigma\cdot\frakn)_i =  \sigma_{ij} \frakn^j,
\end{gathered}
\]
$\frakn$ is the unit normal to the boundary and $\calR$ is a prescribed $(4,0)$-tensor. Moreover, under what conditions in the solution unique?
\item {\bfseries Representation of stresses by stress potentials}:
For $\sigma:\D\to\R^{d\times d}$ satisfying $\div\sigma=0$,  find a $(4,0)$-tensor $\psi$, such that
\[
\sigma = \div\div\psi,
\]
where
\[
(\div\div\psi)_{ij} = \sum_{k,l=1}^d \partial_{kl} \psi_{ikjl}.
\]
Moreover, what is the inherent gauge freedom in the choice of $\psi$?
\end{enumerate}

We consider generalizations of all three problem to the realm of Riemannian manifolds $(\M,\g)$ having constant sectional curvature $\kappa\in\R$, i.e., the $(4,0)$-Riemann curvature tensor $\Rm_\g$ can be represented in local coordinates as
\[
(\Rm_\g)_{ijkl} = \kappa \brk{\g_{ik}\g_{jl} - \g_{jk}\g_{il}}, 
\] 
under general topologies and Sobolev regularity. 

The motivation for considering those problems in a Riemannian setting comes from the theory of incompatible elasticity (also known as anelasticity \cite{Eck48,Kon49}, or non-Euclidean elasticity \cite{ESK09a}), a theory first introduced in the 1950s to model pre-stressed materials.  Mathematically, a pre-stressed material is modeled as a $d$-dimensional Riemannian manifold $(\M,\g)$, where the metric $\g$ encodes the infinitesimal rest lengths  between adjacent material points. For each configuration $f:\M\to\R^d$ corresponds an elastic energy, penalizing for metric deformations. If $(\M,\g)$ is not Euclidean, then the infimal energy over all configurations is non-zero, hence the body is stressed even at equilibrium.  

\paragraph{The Saint-Venant problem}
Denote by $\Theta^1(\M)$ the space of symmetric $(2,0)$-tensors on $\M$. The Saint-Venant problem is the following: what are necessary and sufficient conditions for $\sigma\in\Theta^1(\M)$ to be a Lie derivative of the metric, namely, in local coordinates,
\[
\sigma_{ij} = \nabg_i\omega_j + \nabg_j\omega_i \equiv (\calL_Y\g)_{ij}, 
\]
where $\omega$ is a 1-form, $Y = \omega^\#$ is the corresponding vector field, $Y^i = \g^{ij}\omega_j$, and $\nabg$ is the (Levi-Civita) covariant derivative.

The importance of this problem goes beyond the theory of elasticity: it was first recognized by Berger and Ebin \cite{BE69} that the image of the Lie derivative operator, $X\mapsto\calL_X\g$, is a component of the decomposition of symmetric tensor fields, and whose orthogonal component is the kernel of the divergence operator for tensors. This kernel appears abundantly throughout mechanics.  

The Saint-Venant problem itself was considered by several authors under different assumptions on topology, geometry and regularity. In a series of works, Ciarlet, Geymonat and co-workers (\cite{CCGK07,GK09} and references therein) addressed this question for 
three-dimensional Euclidean domains under $L^2$-regularity and Lipschitz boundary. Yavari and Angoshtari (\cite{Yav13,YA16} and references therein) show how similar results (in a flat setting) can be obtained using the Hodge decomposition for scalar differential forms \cite{Sch95b}. Calabi \cite{Cal61} provided an answer in the smooth category for closed, simply-connected manifolds having constant sectional curvature. Gasqui and Goldschmidt \cite{GG88} improved Calabi's result by generalizing to closed, simply-connected symmetric spaces.  

The smooth version of our first theorem is:

\begin{quote}
{
Let $(\M,\g)$ have constant sectional curvature. Then $\sigma\in\Theta^1(\M)$ satisfies
\[
\sigma = \calL_Y\g
\]
for some vector field $Y$ in $\M$,
if and only if
\[
\bHg\sigma = 0
\Textand
\sigma \perp_{L^2} \calS\BH_\NN^1(\M),
\]
where in local coordinates,
\beq
\begin{split}
(\bHg\sigma)_{ijkl} &= \brk{\nabg_{ik}\sigma_{jl} - \nabg_{jk}\sigma_{il}  - \nabg_{il}\sigma_{jk}  + \nabg_{jl}\sigma_{ik}} \\
&- \kappa\brk{\g_{ik}\sigma_{jl} - \g_{jk}\sigma_{il}  - \g_{il}\sigma_{jk}  + \g_{jl}\sigma_{ik}}
\end{split}
\label{eq:explicit_H_constant}
\eeq
is a generalization of the curl-curl operator, and $\calS\BH_\NN^1(\M)\subset \Theta^1(\M)$ is a finite-dimensional  module of smooth sections, which will be defined in the next section. If $\M$ is simply-connected (as in the settings of \cite{Cal61} and \cite{GG88}), then $ \calS\BH_\NN^1(\M) = \{0\}$.
}
\end{quote}

A similar theorem for three-dimensional Euclidean domains and arbitrary topologies was proved in \cite{GK09}. The space $\calS\BH_\NN^1(\M)$ (denoted by $\bbK$) was however not recognized as finite-dimensional. 

\paragraph{Linearized stress equations}
Our second result addresses boundary-value problems of the form
\beq
\begin{gathered}
\deltag\sigma = 0
\qquad  
\bHg\sigma=\calR 
\qquad
\text{in $\M$} \\
i_{\frakn} \sigma = \tau
\qquad
\text{on $\dM$}.
\end{gathered}
\label{eq:intro_BVP}
\eeq
The operator $\deltag$ is a covariant divergence, which
in a local orthonormal frame takes the form 
\[
(\deltag\sigma)_i = -\sum_{i=1}^d \nabg_j \sigma_{ij}.
\]
The source term $\calR$ is an algebraic curvature \cite{Lee18}, namely a $(4,0)$-tensor satisfying the symmetries pertinent to curvature tensors, which in local coordinates are
\[
\calR_{ijkl} = - \calR_{jikl} =  \calR_{klij}
\Textand
\calR_{ijkl} + \calR_{iklj} + \calR_{iljk} = 0.
\]
The operator $i_\frakn$ is the contraction with the normal to the boundary, $(i_{\frakn} \sigma)_i = \sigma_{ij} \frakn^j$.
The boundary source term $\tau$ is a 1-form restricted to the boundary. 
Such systems arise in linearized theories of elasticity (i.e., in the small-strain limit), with $\sigma$ being the stress tensor;  see \cite{Gur72,Yav13,YA16} in locally-Euclidean setting and \cite{ESK09a,MSK14} in a Riemannian setting. 

We prove the following existence and uniqueness result:

\begin{quote}
{
Let $(\M,\g)$ have constant sectional curvature. Consider
the space of smooth Killing 1-forms,
\[
K(\M)=\BRK{\omega\in\Omega^1(\M)~:~\calL_{\omega^{\sharp}}\g=0}.
\]
The boundary-value problem \eqref{eq:intro_BVP} is solvable  if and only if
\[
\calR\in\image\bHg,
\]
and
\[
\int_{\dM} (\tau,\omega)_\g\, \VolumeD=0
\qquad 
\forall\omega\in K(\M),
\]
where $\gD$ is the induced metric of the boundary and $\VolumeD$ is the corresponding area form.
The solution to \eqref{eq:intro_BVP} is unique up to an element in the finite-dimensional space $\calS\BH^1_{\NN}(\M)$. 
}
\end{quote}
As we show, the condition $\calR\in\image{\bHg}$ rises naturally in applications: most prominently, the Riemannian curvature tensor $\Rm_\g$ satisfies this condition. If $\calR$ and $\tau$ are Sobolev sections, then the solution inherits the regularity with appropriate estimates on the Sobolev norm of $\sigma$. We note that $K(\M)$ is finite-dimensional, and every weak killing field is in fact smooth as a result of Korn's inequality \cite[Ch.~5.12]{Tay11a}.
We further observe that if \eqref{eq:intro_BVP} is solvable, then $\sigma$ is a solution of a regular elliptic system,
constituting a generalization of the biharmonic equations for the stress/strain field in classical elasticity \cite[p.~133]{Gur72}, supplemented by a complete set of boundary conditions and a uniqueness clause. 

\paragraph{Representation of stresses by stress potentials}
Our third result addresses the existence of stress potentials: suppose that  $\sigma\in\Theta^1(\M)$ satisfies
\[
\deltag\sigma=0.
\]
If $(\M,\g)$ is a simply-connected Euclidean domain, then $\sigma_{ij} =-\sum_{k=1}^d \partial_k Q_{ikj}$ for a $(3,0)$-tensor $Q$, anti-symmetric in its first two indices. A classical procedure in elasticity hinges on the observation  that one can further choose $Q$ such that $\sum_{j=1}^d \partial_j Q_{ikj}=0$, and deduce the existence of a stress potential $\psi$, which is an algebraic curvature, such that $Q_{ikj} = -\sum_{l=1}^d \partial_l \psi_{ikjl}$ \cite{Tru59, Gur72, GK06, CCGK07, Yav13, YA16} (the applications are restricted to dimensions 2,3, however the existence of a potential holds in any dimension). Thus,
\[
\sigma_{ij} = \sum_{k,l=1}^d \partial_{kl} \psi_{ikjl}.
\]
The choice of a potential is non-unique, calling for a choice of gauge \cite{Max70,Mor92,Pom15}. 

The smooth version of our theorem is:

\begin{quote}
{
Let $(\M,\g)$ have constant sectional curvature and 
let $\sigma\in\Theta^1(\M)$ satisfy $\deltag\sigma=0$. 
Suppose that there exists an algebraic curvature $\eta$, satisfying
\[
\begin{gathered}
\sigma - \bHg^*\eta \perp \calS\BH^1_\NN(\M) 
\Textand
i_{\frakn}(\sigma - \bHg^*\eta) = 0,
\end{gathered} 
\]
where $\bHg^*$ is the $L^2$-dual of $\bHg$, and the orthogonality condition is with respect to the $L^2$ inner-product.
Then, there exists an algebraic curvature $\psi$ satisfying
\[
\bHg^*\psi=\sigma.
\]
}
\end{quote}

If $\sigma$ and $\eta$ are Sobolev sections, then $\psi$ inherits the regularity with appropriate estimates on its Sobolev norm. If $i_{\frakn}\sigma=0$ and $\sigma$ is orthogonal to $\S\BH^1_{\NN}(\M)$, then the conditions are satisfied trivially for $\eta=0$.  
We further identify gauge and boundary conditions that can be imposed on $\psi$.

All three problems are solved using an elliptic theory  \cite{KL21a} pertinent to a class of vector-valued forms known as double forms \cite{deR84,Cal61,Gra70,Kul72}. Supplementing $\bHg$ and $\bHg^*$ with another second-order operator $\bFg$ and its dual $\bFg^*$, we define a fourth-order differential operator, $\Bg:\Theta^1(\M)\to\Theta^1(\M)$,  defined by
\[
\Bg=\bHg^*\bHg+\bHg\bHg^*+ \bFg^*\bFg +\bFg\bFg^*,
\]
reminiscent of how the exterior derivative and its dual give rise to the Hodge
laplacian in the classical theory of scalar differential forms \cite{Tay11a,Sch95b}. 
We prove the regular ellipticity of $\Bg$ under several sets of boundary conditions. This
elliptic theory manifests analogies to the theory of the Hodge laplacian, but cannot be derived
from it. In particular, under certain exactness conditions, which are satisfied in the case of manifolds having constant sectional curvature, this elliptic theory leads to a symmetries-preserving decomposition of double forms, and in particular of $\Theta^1(\M)$.

The curl-curl operator and its duals have been studied quite extensively (e.g., \cite{Cal61, Gur72, GG88, CCGK07, GK09}). 
To the best of our knowledge, earlier works do not recognize these operators as the progenitors of an elliptic theory of a bilaplacian operator. In the Euclidean setting, and in two and three dimensions, the primary analytical apparatus used in the study of partial differential systems featuring the curl-curl operator is a double iteration of first order methods \cite{Gur72, CCGK07, GK09, Yav13, YA16}. Such methods are classically designed for vector fields and scalar differential forms, and the reason they extend for symmetric tensor fields relies heavily on the fact that the space in question is Euclidean. Thus, the most immediate shortcoming of this approach is that it is not suitable in non-Euclidean settings. Other shortcomings, even in Euclidean settings, concern the ability to impose full boundary conditions, and  exploiting gauge freedom. In this paper, we demonstrate how the elliptic theory of bilaplacians developed in \cite{KL21a} resolves these matters.

The current work is limited to manifolds having constant sectional curvature. The generalized curl-curl operator $\bHg$ has a natural generalization to arbitrary Riemannian manifolds, given in a local orthonormal frame by
\beq
\begin{split}
(\bHg\sigma)_{ijkl} &= \half \brk{\nabg_{ik}\sigma_{jl} - \nabg_{jk}\sigma_{il}  - \nabg_{il}\sigma_{jk}  + \nabg_{jl}\sigma_{ik}} \\
&+ \half \brk{\nabg_{ki}\sigma_{jl} - \nabg_{kj}\sigma_{il}  - \nabg_{ji}\sigma_{jk}  + \nabg_{lj}\sigma_{ik}} \\
&- \half \sum_{s=1}^d \brk{(\Rm_\g)_{sijk}\sigma_{sl} - (\Rm_\g)_{sijl}\sigma_{sk} - (\Rm_\g)_{sjik}\sigma_{sl} + (\Rm_\g)_{skli}\sigma_{sj}}.
\end{split}
\label{eq:explicit_H_general}
\eeq
This operator, however, does not annihilate Lie derivatives of the metric, which is a key property in the aforementioned exactness conditions. In locally-symmetric spaces, $\bHg$ only annihilates Hessians of scalar functions, which form a subclass of Lie derivatives of the metric. 
Note that \cite{GG88}, solves the Saint-Venant problem for a compact, simply-connected symmetric space without boundary. 
In this work, $\bHg$ is replaced by an operator $\alpha\circ\bHg$, which annihilates Lie derivatives in locally-symmetric spaces, where $\alpha$ in a smooth bundle map from the space of algebraic curvatures to itself. The construction of $\alpha$ is specific  to the homogeneous structure, and more importantly, it is not clear whether $\alpha\circ\bHg$ differs from $\bHg$ by a lower-order differential operator, which is required in order to fit into our elliptic theory.

\paragraph{The structure of this paper}
In Section~2 we survey double forms along with the main results obtained in \cite{KL21a}.
To keep the survey as concise as possible, we state all formulas and definitions without reference to regularity. Detailed constructions, proofs and further references are found in \cite{KL21a}.
In Section~3, we construct the differential operators $\bHg$, $\bFg$ and their duals for general manifolds, and show they possess a particular set of properties in the case of manifolds having constant sectional curvature.
In Section~4 we address the Saint Venant problem. 
In Section~5 we  present the equations of incompatible elasticity, which to a large extent are the motivation to this work, and address the existence and uniqueness of solutions.
In Section~6 we address the existence of stress potentials.


\paragraph{Acknowledgments}
We thank Cy Maor for pointing out the relation between the existence of stress potentials and the Saint-Venant problem. We thank Cy Maor and Asaf Shachar for their comments on the manuscript. We thank the anonymous reviewers for many valuable suggestions, and for helping us to improve the readability of this article.
This research was partially supported by the Israel Science Foundation Grant No. 1035/17.

\section{Double forms}

Let $(\M,\g)$ be a Riemannian manifold.
We consider the spaces of sections, 
\[
\VectorFormsKM = \Omega^k(\M;\Lambda^{m}T^*\M)= \Gamma(\Lambda^{k,m}T^*\M),
\]
known as double forms, or  $(k,m)$-forms,
where
\[
\Lambda^{k,m}T^*\M=\Lambda^kT^*\M\otimes\Lambda^{m}T^*\M.
\]
Double forms are differential forms taking values in spaces of differential forms.
$(k,m)$-forms differ from $(k,m)$-tensors in two aspects: the $m$-part of a $(k,m)$-form is covariant rather than contravariant, and both $k$- and $m$-part of a $(k,m)$-form have alternating symmetry.
We will commonly refer to the $k$-part of a $(k,m)$-form as a the ``form part" and to the $m$-part as the ``vector part".
In local coordinates, a $(k,m)$-form has the following representation,
\[
\psi = \psi_{i_1,\dots,i_k; j_1,\dots,j_m} 
\brk{dx^{i_1} \wedge\dots \wedge dx^{i_k}}
\otimes
\brk{dx^{j_1} \wedge\dots \wedge dx^{j_m}},
\] 
with summation over increasing indices. 
Double forms along with corresponding algebraic and differential operators were introduced and addressed in 
 \cite{deR84,Cal61,Gra70,Kul72}. This section presents a concise survey, along with results obtained in 
\cite{KL21a}.

The vector bundle $\Lambda^{k,m}T^*\M$ has a natural graded wedge product,
\[
\wedge:\Lambda^{k,m}T^*\M\times\Lambda^{n,\ell}T^*\M\to \Lambda^{k+n,m+\ell}T^*\M,
\] 
defined by the linear extension of
\[
(\omega\otimes F)\wedge(\alpha\otimes Q)=(\omega\wedge\alpha)\otimes(F\wedge Q),
\]
turning $\Lambda^{*,*}T^*\M = \bigoplus_{k,m}\Lambda^{k,m}T^*\M$ into a graded algebra. 
The graded algebra of double forms \cite{deR84,Cal61,Gra70,Kul72} is defined as
\[
\VectorFormsM=\bigoplus_{k,m}\VectorFormsKM.
\]

Double forms have a natural tensorial involutive operation of flip, or transposition,
$(\cdot)^T:\VectorFormsKM\to \Omega^{m,k}(\M)$ defined by 
\[
\psi^T(Y_1,\dots, Y_m;X_1,\dots ,X_k)=\psi(X_1,\dots ,X_k;Y_1,\dots ,Y_m),
\]
where the semicolon separates between the arguments of the form and vector parts. A $(k,k)$-form $\psi$ satisfying $\psi^T=\psi$ is called symmetric. The space of symmetric forms is denoted by $\Theta^k(\M)$. 
Metrics, Ricci tensors and Hessians of scalar functions can be viewed as symmetric $(1,1)$-forms, whereas the $(4,0)$-versions of Riemannian curvature tensors can be viewed as symmetric $(2,2)$-forms. For a metric $\g$, the relation between the $(2,2)$-form $\Rm_\g$ and the Riemannian $(3,1)$-endomorphism $R_\g\in\Omega^2(\M;\End(T\M))$ is 
\[
\Rm_\g(X,Y;Z,W) = (R_\g(X,Y)Z,W)_\g.
\]
Another version of the curvature tensor is the curvature operator $\calR_\g\in\End(\Lambda^2(T\M))$ \cite[p.~83]{Pet16}, related to $\Rm_\g$ via
\[
(\calR_\g(X,Y),Z\wedge W)_\g = \Rm_\g(X,Y;W,Z).
\]

We denote by $W^{s,p}\Omega^{k,m}(\M)$ and $W^{s,p}\Theta^k(\M)$ the Sobolev versions of these spaces, for $s\in\mathbb{N}\cup\{0\}$ and $p\geq 2$ (defined by the completion of $\Omega^{k,m}(\M)$ with respect to the $W^{s,p}$ norms, which in turn are defined using covariant differentation). These are equipped with the $L^2$-inner-product
\[
\bra\psi,\eta\ket = \int_\M (\psi,\eta)_\g\,\VolumeG,
\]
where $\VolumeG$ is the Riemannian volume form.
 
Double forms, like any other vector-valued form, are equipped with a Hodge-dual isomorphism
$\starG:\Lambda^{k,m}T^*\M\rightarrow\Lambda^{d-k,m}T^*\M$, 
defined by the linear extension of
\[
\starG(\omega\otimes F) = \starG\omega \otimes F.
\]
For a tangent vector $v\in T\M$, the interior product $i_v:\Lambda^{k,m}T^*\M\to\Lambda^{k-1,m}T^*\M$ is defined by the linear extension
\[
i_v(\omega\otimes F) = i_v\omega \otimes F.
\]
Another algebraic operation on double forms is the Bianchi sum, which is the bundle map $\G:\Lambda^{k,m}T^*\M\rightarrow  \Lambda^{k+1,m-1}T^*\M$, given by
\[
\G\psi(X_1,...,X_{k+1};Y_1,...,Y_{m-1})=\sum_{j=1}^{k+1}(-1)^{j+1}\psi(X_1,...,\hat{X}_j,...,X_{k+1};X_j,Y_1,...,Y_{m-1}),
\]
where as usual, $\hat{X}_j$ denotes an omitted argument.
It is noteworthy that $\sigma\in\Omega^{1,1}(\M)$ is symmetric if and only if $\frakG\sigma=0$. 
Sections $\psi\in\Omega^{2,2}(\M)$ satisfying $\frakG\psi=0$ are called algebraic curvatures, and are in particular symmetric.
The first Bianchi identity satisfied by the Riemannian curvature tensor can be written as $\G\Rm_\g=0$.

Since the vector part of $\Lambda^{k,m} T^*\M$  is also an exterior algebra, all the operators acting on vector-valued forms can be defined on the vector part through transposition. 
The symbol ``V" will be used to denote operators acting on the exterior algebra of the vector part; 
explicitly,
\beq
\begin{gathered}
\starG^V \psi = (\starG \psi^T)^T
\qquad
i_X^V \psi = (i_X \psi^T)^T 
\Textand
\G_V\psi=(\G\psi^T)^T.
\end{gathered}
\label{eq:dual_ops}
\eeq
The Hodge-dual operators yield an isometry,
\[
\starG\starG^V:\Theta^k(\M)\rightarrow \Theta^{d-k}(\M),
\]
which restricts to a $W^{s,p}$-isometry for all $s\in\mathbb{N}\cup\{0\}$ and $p\geq 2$. 

Another tensorial operator acting on double forms is the metric contraction,
\[
\trace_\g: \Lambda^{k,m} T^*\M\to \Lambda^{k-1,m-1} T^*\M,
\]
which can be written in terms of an orthonormal frame $\ortframe$ of $T\M$,
\[
\trace_\g = \sum_{i=1}^d i_{E_i} i_{E_i}^V.
\]
Its metric dual is $\g\wedge : \Lambda^{k-1,m-1} T^*\M\to \Lambda^{k,m} T^*\M$.

We next consider first-order differential operators.
The covariant  exterior derivative
\[
\dg:\Omega^{k,m}(\M)\to\Omega^{k+1,m}(\M)
\]
is an $\R$-linear graded operator, defined by the linear extension of \cite[pp.~60, 362]{Pet16}
\[
\dg(\omega\otimes F) = d\omega\otimes F + (-1)^k \omega\wedge \nabg F,
\]
where $\nabg$ is the Riemannian connection and $\nabg F$ is viewed here as a $(1,m)$-form.  In a covariant notation, 
\[
\dg\psi(X_{1},\dots,X_{k+1};Y_{1},\dots,Y_{m})=\sum_{i=1}^{k+1}(-1)^{k+1}\nabg_{X_{i}}\psi(X_{1},\dots,\hat{X}_{i}\dots,X_{k+1};Y_{1},\dots Y_{m}).
\]
Unlike the exterior derivative, $\dg\dg$ is in general not zero, and is related to the curvature endomorphism of $R_\g^* \in\Omega^2(\M;\End(\Lambda^m T^*\M))$, which is related in turn to $R_\g$ via the Ricci identity \cite[p.~205]{Lee18}. Since the metric $\g$ is parallel, $\dg\g=0$; moreover, the second Bianchi identity reads $\dg\Rm_\g=0$.

The $L^2$-dual of $\dg$ \cite{Kul72} is denoted by 
\[
\deltag: \Omega^{k+1,m}(\M)\to\Omega^{k,m}(\M).
\] 
For an orthonormal frame $\ortframe$ of $T\M$, 
\[
\deltag = - \sum_{i=1}^d i_{E_i}\nabg_{E_i}.
\]
The operators $\dg$ and $\deltag$ can be viewed as generalized ``curl" and ``div" operators, acting on double forms.
The vector counterparts of the first-order operators are denoted by 
\[
\dgV\psi = (\dg\psi^T)^T
\Textand
\deltagV\psi = (\deltag\psi^T)^T.
\]

For a scalar function $f\in \Omega^{0,0}(\M)$,
\[
\dg\dgV f=\dgV\dg f=\nabg (df)^T=\Hess_\g f\in \Omega^{1,1}(\M).
\]
For general values of $k$ and $m$, $\dg$ and $\dgV$ do not commute. For example, let $f\in \Omega^{0,0}(\M)$ be a scalar function, then
\[
\dgV\dg df = 0,
\]
whereas
\[
\dg\dgV df = R_\g^*\circ (df)^T.
\]

We next introduce second-order differential operators on double forms,
\[
\begin{aligned}
&\Hg : \VectorFormsKM\to\Omega^{k+1,m+1}(\M) 
&\qquad
&\Hg^* : \VectorFormsKM\to\Omega^{k-1,m-1}(\M) \\
&\Fg : \VectorFormsKM\to\Omega^{k+1,m-1}(\M)
&\qquad
&\Fg^* : \VectorFormsKM\to\Omega^{k-1,m+1}(\M) ,
\end{aligned}
\]
defined by
\[
\begin{aligned}
&\Hg = \smallhalf(\dgV\dg+\dg\dgV)                     
&\qquad
&\Hg^* = \smallhalf(\deltag\deltagV+\deltagV\deltag)   \\
&\Fg = \smallhalf(\dg\deltagV+\deltagV\dg)
&\qquad
&\Fg^* = \smallhalf(\dgV\deltag+\deltag\dgV).  
\end{aligned}
\]
As the notation suggests, $\Hg$ and $\Hg^*$, and $\Fg^*$ and $\Fg$ are mutually dual with respect to the $L^2$ inner-product. Moreover, for $\psi\in\Omega^{k,m}(\M)$,
\[
\begin{aligned}
\Hg^*\psi &= (-1)^{dk+dm}\starG \starG^V \Hg \starG \starG^V\psi \\
\Fg^*\psi &= (-1)^{dk+d+1} \starG \Hg \starG \psi \\
\Fg\psi &= (-1)^{dm+d+1} \starG^V \Hg \starG^V\psi.
\end{aligned}
\]
The operator $\Hg$ commutes with transposition, 
\[
(\Hg\psi^T)^T = \Hg\psi,
\]
and by duality, so does $\Hg^*$. 
On the other hand,
\[
(\Fg\psi^T)^T = \Fg^*\psi.
\]
The operators $\Hg$ and $\Hg^*$ can be viewed as generalized ``curl-curl" and ``div-div" operators, whereas $\Fg$ and $\Fg^*$ are mixed combinations of ``curl" and ``div". In Euclidean space $\Hg$ restricted to $(1,1)$-forms coincides with the $\curl\curl$ operator presented in the introduction. 

In this paper we consider restrictions of these operators to symmetric forms.  
Consider the diagram
\[
\begin{tikzcd}
&& {\Theta^{k+1}(\M)} \\
&& {} \\
{}&& {\Theta^{k}(\M)} && {\Omega^{k+1,k-1}(\M)} \\
&& {} \\
&& {\Theta^{k-1}(\M)}
\arrow["{\Hg}"', curve={height=12pt}, from=5-3, to=3-3]
\arrow["{\Hg}"', curve={height=12pt}, from=3-3, to=1-3]
\arrow["{\Hg^*}", curve={height=-12pt}, tail reversed, no head, from=3-3, to=1-3]
\arrow["{\Hg^*}"', curve={height=12pt}, from=3-3, to=5-3]
\arrow["{\tfrac12(\Fg^* + (\Fg^*(\cdot))^T)}", curve={height=-12pt}, from=3-5, to=3-3]
\arrow["{\Fg}", curve={height=-12pt}, from=3-3, to=3-5]
\end{tikzcd}
\]
In a locally-Euclidean setting, the second-order differential operators satisfy exactness conditions
\[
\begin{gathered}
\Hg\Hg=0 \qquad \Fg\Hg=0 \qquad \Hg(\Fg^*+(\Fg^*(\cdot))^T)=0 \\
\Hg^*\Hg^*=0 \qquad \Fg\Hg^*=0 \qquad \Hg^*(\Fg^*+\Fg^*(\cdot)^T)=0. 
\end{gathered}
\]
In a general Riemannian setting, these exactness conditions do not hold. In the next section, we show how to modify the second-order differential operators by tensorial terms, such to retain the exactness conditions in manifolds having constant sectional curvature.
 
Generally, assume graded tensorial operators 
\[
\begin{aligned}
& D_\g:\Lambda^{k,m}T^*\M\rightarrow \Lambda^{k+1,m+1}T^*\M \\ 
& S_\g:\Lambda^{k,m}T^*\M\rightarrow \Lambda^{k+1,m-1}T^*\M, 
\end{aligned}
\]
along with their metric duals
\[
\begin{aligned}
& D_\g^*:\Lambda^{k+1,m+1}T^*\M\rightarrow \Lambda^{k,m}T^*\M \\ 
& S_\g^*:\Lambda^{k+1,m-1}T^*\M\rightarrow \Lambda^{k,m}T^*\M, 
\end{aligned}
\]
satisfying the same respective symmetries as $\Hg$ and $\Fg$, 
\[
(D_\g\psi^T)^T = D_\g\psi
\Textand 
(S_\g\psi^T)^T = S_\g^*\psi.
\]
We introduce the corresponding families of operators $\bHg:\Omega^{k,m}(\M)\rightarrow\Omega^{k+1,m+1}(\M)$ and  $\bFg:\Omega^{k,m}(\M)\rightarrow\Omega^{k+1,m-1}(\M)$, along with their $L^2$-duals,
\[
\begin{aligned}
& \bHg=\Hg+  D_\g        &\qquad\qquad  &\bHg^*=\Hg^*+  D_\g^* \\
& \bFg=\Fg+  S_\g  &  \qquad\qquad & \bFg^*=\Fg^*+  S_\g^*.
\end{aligned}
\]
By construction, $\bHg$ and $\bHg^*$ commute with transposition and $(\bFg(\cdot)^T)^T = \bFg^*$. 

The physical context imposes an analysis on manifolds with boundaries. Like in second-order elliptic theory, natural boundary conditions arise from integration by parts formulas. To this end, we introduce mixed projections of tangential and normal boundary components,
\beq
\begin{aligned}
&\PttD:\Omega^{k,m}(\M)\to \Omega^{k,m}(\dM) 
&\qquad 
&\PntD:\Omega^{k,m}(\M)\to \Omega^{k-1,m}(\dM) \\ 
&\PtnD:\Omega^{k,m}(\M)\to \Omega^{k,m-1}(\dM) 
&\qquad 
&\PnnD:\Omega^{k,m}(\M)\to \Omega^{k-1,m-1}(\dM).
\end{aligned}
\label{eq:intro_zero_order_operators}
\eeq
The first superscript in $\frakt\frakt$, $\frakt\frakn$, $\frakn\frakt$ and $\frakn\frakn$ refers to the projection of the form part, whereas the second superscript refers to the projection of the vector part. Thus, 
$\PttD\psi$ is obtained by pulling back $\psi$ to the boundary, 
$\PntD\psi$ is obtained by pulling back $i_\frakn\psi$ to the boundary, where $\frakn$ is the unit normal to the boundary, 
$\PtnD\psi$ is obtained by pulling back $i_\frakn^V\psi$ to the boundary, and finally
$\PnnD\psi$ is obtained by pulling back $i_\frakn i_\frakn^V\psi$ to the boundary.

The boundary projection operators satisfy the duality relations
\[
\begin{aligned}
&\PttD  \starG\psi  =  (-1)^{d+1} \star_{\gD} \PntD \psi
&\qquad
&\PtnD  \starG\psi  = (-1)^{d+1}  \star_{\gD} \PnnD \psi \\
&\PntD\starG\psi  = (-1)^{d+k+1} \star_{\gD} \PttD \psi
&\qquad
&\PnnD\starG\psi = (-1)^{d+k+1}  \star_{\gD} \PtnD \psi \\
&\PttD  \starG^V\psi  =  (-1)^{d+1} \star_{\gD}^V \PtnD \psi
&\qquad
&\PntD  \starG^V \psi  = (-1)^{d+1}  \star_{\gD}^V \PnnD \psi \\
&\PtnD\starG^V\psi  = (-1)^{d+m+1} \star_{\gD}^V \PttD \psi
&\qquad
&\PnnD\starG^V\psi = (-1)^{d+m+1}  \star_{\gD}^V \PntD \psi,
\end{aligned}
\]
where $\gD$ is the pullback metric at the boundary.

We further introduce first-order differential boundary operators,
\beq
\begin{aligned}
\frakT\psi &= \smallhalf \brk{\PntD \dg\psi-\dgD \PntD\psi}+\smallhalf \brk{\PtnD\dgV\psi- \dgD \PtnD\psi} \\
\frakT^*\psi &=  -\smallhalf\brk{ \PtnD\deltag\psi+\deltagD \PtnD\psi} - \smallhalf\brk{\PntD\deltagV\psi+\deltagDV\PntD\psi} \\
\frakF^*\psi &= \smallhalf\brk{\PnnD \dgV\psi- \dgDV \PnnD \psi} -\smallhalf\brk{\PttD \deltag \psi+\deltagD \PttD  \psi} \\
\frakF\psi &= \smallhalf\brk{\PnnD \dg\psi-  \dgD \PnnD \psi} -\smallhalf\brk{\PttD\deltagV \psi+ \deltagDV \PttD  \psi},
\end{aligned}
\label{eq:1st_order_proj}
\eeq
which satisfy the duality relations
\[
\begin{split}
& \frakT^*= (-1)^{dk+dm + k +m} \star_{\gD}\star_{\gD}^V\frakT\starG\starG^V \\
& \frakF^*=(-1)^{dk+k} \star_{\gD}\frakT\starG\\
& \frakF= (-1)^{dm + m} \star_{\gD}^V\frakT\starG^V.
\end{split}
\]
The boundary operators satisfy the integration by parts formulas, 
\beq
\begin{aligned}
\bra \bHg\psi, \eta\ket &= 
\bra \psi, \bHg^*\eta\ket + 
\int_{\dM}\Brk{(\PttD\psi,\frakT^* \eta)_\gD - 
(\frakT\psi,\PnnD \eta)_\gD}\VolumeD  \\
\bra \bFg\psi, \eta\ket &= 
\bra \psi, \bFg^*\eta\ket + 
\int_{\dM}\Brk{(\PtnD\psi,\frakF^*\eta)_\gD - 
(\frakF\psi,\PntD\eta)_\gD}\VolumeD,
\end{aligned}
\label{eq:ibp}
\eeq
independently of the choices of $D_\g$ and $S_\g$.

We introduce the fourth-order linear differential operator  $\Bg:\Omega^{k,m}(\M)\rightarrow \Omega^{k,m}(\M)$,
\[
\Bg = \bHg \bHg^* + \bHg^* \bHg + \bFg^* \bFg + \bFg \bFg^*,
\]
which we call a double bilaplacian. 
A substantial part of \cite{KL21a} is the establishment of the regular ellipticity of $\Bg$ when supplemented with several sets of boundary conditions. 

Let $E,E_1,\dots,E_m$ be vector spaces and let $L_i:E\to E_i$, $i=1,\dots,m$, be linear operators. As standard in the PDE literature, we denote
\[
(L_1,\dots,L_m) : E\to \bigoplus_{i=1}^m E_i.
\] 
In the framework of symmetric forms, 
\[
\Bg:\Theta^k(\M) \to \Theta^k(\M), 
\]
and the operators 
\[
\begin{aligned}
\frakB_\TT &= (\Bg,\PttD,\PtnD,\frakT,\frakF,\PttD\bHg^*,\frakT\bHg^*) \\
\frakB_\NN &= (\Bg,\PnnD,\PtnD,\frakT^*,\frakF,\PnnD\bHg,\frakT^*\bHg)
\end{aligned}
\label{eq:boundary_data_symmetric}
\] 
define regular elliptic boundary-value problems for elements in $\Theta^k(\M)$.
The associated kernels---the biharmonic modules,
\beq
\begin{gathered}
\calS\BH^k_{\TT}(\M) = \ker (\bHg,\bHg^*,\bFg,\PttD,\PtnD,\frakT,\frakF) \\
\calS\BH^k_{\NN}(\M) = \ker (\bHg,\bHg^*,\bFg,\PnnD,\PtnD,\frakT^*,\frakF),
\end{gathered}
\label{eq:BHTTNN}
\eeq
are finite-dimensional and contains only smooth sections, with an isometry,
\[
\starG\starG^V:\calS\BH^k_{\TT}(\M)\rightarrow \calS\BH^{d-k}_{\NN}(\M).
\]   

\section{Symmetric forms in constant curvature}

Following \cite[Sec.~7.3]{KL21a}, we focus on the following diagram, centered around symmetric $(1,1)$-forms: 
\beq
\begin{tikzcd}
&& {\Theta^2(\M)} \\
&& {} \\
{} && {\Theta^1(\M)} && {\Omega^{2,0}(\M)} \\
&& {} \\
&& {\Theta^0(\M)}
\arrow["{\bHg}"', curve={height=12pt}, from=5-3, to=3-3]
\arrow["{\bHg}"', curve={height=12pt}, from=3-3, to=1-3]
\arrow["{\bHg^*}", curve={height=-12pt}, tail reversed, no head, from=3-3, to=1-3]
\arrow["{\bHg^*}"', curve={height=12pt}, from=3-3, to=5-3]
\arrow["{\tfrac12(\bFg^* + (\bFg^*(\cdot))^T)}", curve={height=-12pt}, from=3-5, to=3-3]
\arrow["{\bFg}", curve={height=-12pt}, from=3-3, to=3-5]
\end{tikzcd}
\label{eq:exact_diagram}
\eeq
By appropriate choices of $D_\g$ and $S_\g$, we construct $\bHg$ and $\bFg$ satisfying exactness conditions in the case where $\g$ has constant sectional curvature, namely, the Riemannian curvature tensor is of the form
\[
\Rm_\g  = \tfrac12\kappa \, \g\wedge \g,
\]
where $\kappa$ is a constant.

The following properties of double forms in manifolds having constant sectional curvature are easily verified from the formulas in \cite[Sec.~3]{KL21a}:

\begin{lemma}
Let $(\M,\g)$ have constant section curvature $\kappa$, then for every $\psi\in\Omega^{k,m}(\M)$,
\begin{gather}
\dg\dg\psi = -\kappa \, \g\wedge \G\psi \label{eq:kappa1} \\
\deltag\deltag\psi = -\kappa\,\trace_\g\G_V\psi \label{eq:kappa2} \\
\dg\dgV\psi -\dgV\dg\psi = (m-k)\kappa\, \g\wedge\psi \label{eq:kappa3} \\ 
\dg\deltagV\psi -\deltagV\dg\psi = (d-m-k)\kappa\, \G\psi \label{eq:kappa4}.
\end{gather}
In particular, for every $\psi\in\Theta^1(\M)$,
\beq
\dg\dg\psi=0 
\qquad 
\dg\dgV\psi = \dgV\dg\psi
\textand
\dg\deltagV\psi = \deltagV\dg\psi.
\label{eq:kappa5}
\eeq
\end{lemma}

\subsection{Construction of $\bHg$} 

We first construct $\bHg$ without assuming constant sectional curvature. 
We then show that for manifolds having constant sectional curvature,  this construction falls within the setting of \cite{KL21a}; it coincides with the operator denoted by $D_{1}$ in \cite{Cal61}. In \cite{GG88}, where $(\M,\g)$ is a locally-symmetric space, $D_{1}=\alpha\circ\bHg$, where $\alpha:\ker\frakG\rightarrow\ker\frakG$ is a smooth projection map only defined in locally-symmetric spaces. In constant sectional curvature, $\alpha=\id$, but in general the construction of $\alpha$ is abstract and expected to yield a different explicit operation for different symmetric spaces.

We introduce the tensorial operation $D_\g:\Lambda^{k,m}T^*\M\rightarrow \Lambda^{k+1,m+1}T^*\M$ given by
\beq
D_\g\psi = \tfrac12\brk{\trace_\g(\Rm_\g\wedge\psi)- \trace_\g \Rm_\g\wedge\psi-\Rm_\g\wedge(\trace_\g\psi)}.
\label{eq:D}
\eeq

We summarize below some of the properties of $D_\g$.
First, by a direct calculation:

\begin{lemma}
Given an orthonormal frame $\ortframe$ for $T\M$,
\beq
D_\g\psi = \tfrac12 \sum_i((i_{E_i}\Rm_\g)\wedge (i^V_{E_i}\psi)+(i^V_{E_i}\Rm_\g)\wedge (i_{E_i}\psi)).
\label{eq:formula_D}
\eeq
\end{lemma}

\begin{lemma}
The operator $D_\g$ commutes with both transposition and the Bianchi sum,
\[
(D_\g\psi^T)^T = D_\g\psi 
\Textand
D_\g\frakG = \frakG D_\g.
\]
Moreover, for $\psi\in\Omega^{k,m}(\M)$ and $\varphi\in\VectorFormsM$,
\beq
D_\g(\psi\wedge\varphi)=(D_\g \psi)\wedge\varphi+(-1)^{k+m}\psi\wedge(D_\g \varphi).
\label{eq:DgLeibniz}
\eeq
\end{lemma}

\begin{proof}
Commutation with transposition is immediate since both the metric trace and the wedge product commute with transposition and $\Rm_\g$ is symmetric. 
Commutation with the Bianchi sum follows from similar considerations, as $\frakG\Rm_\g=0$.
Finally, the composition rule follows from \eqref{eq:formula_D}.
\end{proof}

\begin{lemma}
For a scalar-valued 1-form, $\omega\in\Omega^{1,0}(\M)$, 
\beq
(D_\g\omega)(X,Y;Z) = - \tfrac12 \omega(R_\g(X,Y)Z).
\label{eq:D_explicit_expression_1_forms}
\eeq
\end{lemma}

\begin{proof}
Substituting $\omega$ into \eqref{eq:formula_D}, noting that $i_{E_i}^V\omega=0$ (since $i_{E_i}$ annihilates $(0,1)$-forms),
\[
\begin{split}
D_\g\omega(X,Y;Z) &= \tfrac12 \sum_i \omega(E_i) \, \Rm_\g(X,Y;E_i,Z) \\
&= - \tfrac12 \sum_i \omega(E_i) \, (R_\g(X,Y)Z,E_i)_\g \\
&= -\tfrac12 \omega(R_\g(X,Y)Z).
\end{split}
\]
\end{proof}

\begin{lemma}
For $\psi\in\Omega^{1,1}(\M)$,
\beq
\begin{split}
D_\g\psi(X,Y;Z,W)&= \tfrac12 \brk{\psi(X;R_\g(Z,W)Y)- \psi(Y;R_\g(Z,W)X)} \\
&+ \tfrac12 \brk{\psi(R_\g(X,Y)W;Z)- \psi(R_\g(X,Y)Z;W)}.
\end{split}
\label{eq:D_explicit_expression}
\eeq
\end{lemma}

\begin{proof}
By bilinearity, it suffices to verify this assertion for $\psi = \omega\wedge \eta$, with $\omega\in\Omega^{1,0}(\M)$ and $\eta\in\Omega^{0,1}(\M)$. By the composition rule \eqref{eq:DgLeibniz},
\[
D_\g\psi = D_\g\omega \wedge \eta - \omega \wedge D_\g\eta,
\]
hence
\[
\begin{split}
D_\g\psi(X,Y;Z,W) &= D_\g\omega(X,Y;Z)\eta(W) - D_\g\omega(X,Y;W)\eta(Z) \\
&\qquad -  \omega(X)   D_\g\eta(Y;Z,W)  + \omega(Y)  D_\g\eta(X;Z,W) \\
&\EwR{eq:D_explicit_expression_1_forms} - \tfrac12\omega(R_\g(X,Y)Z)\eta(W) + \tfrac12\omega(X)\eta(R_\g(Z,W)Y) \\
&\qquad + \tfrac12\omega(R_\g(X,Y)W)\eta(Z) - \tfrac12 \omega(Y)  \eta(R_\g(Z,W)X),
\end{split}
\]
which is the desired result.
\end{proof}

For $k\in\BRK{0,1}$ we define $\bHg:\Omega^{k,k}(\M)\rightarrow\Omega^{k+1,k+1}(\M)$,
\[
\bHg=\Hg+ D_\g.
\]
For functions $D_\g f=0$, hence $\bHg f=\Hg f$. In view of \eqref{eq:D_explicit_expression}, $\bHg$ assumes the form  \eqref{eq:explicit_H_general} with respect to an orthonormal frame. Since both $\Hg$ and $D_\g$ commutes with $\frakG$,
\[
\frakG\bHg = \bHg\frakG.
\]
Moreover, since $\Hg\g=0$ (since $\dg\g=0$ and $\dgV\g=0$), insertion of $\g$ into \eqref{eq:D_explicit_expression} yields
\beq
\bHg\g = D_\g\g = -2\Rm_\g.
\label{eq:H_on_the_metric}
\eeq

We next examine how $\bHg$ restricts to symmetric forms:

\begin{proposition}
Let $\psi\in\Theta^1(\M)$. For sufficiently small $t\in\R$, the tensor $\g(\psi,t)=\g+t\psi$ is a metric. Then,
\beq
\bHg\psi = 2\derivAtZero\Rm_{\g(\psi,t)} + 2\,D_\g\psi.
\label{eq:prop3.5}
\eeq
\end{proposition}

\begin{proof}
This is a direct consequence of the well-known variation formula of the curvature tensor (e.g. \cite[p.~560]{Tay11b}) and expression \eqref{eq:D_explicit_expression} for $D_\g\psi$. 
\end{proof}

Let $U\in\VF(\M)$. We denote by $U^\flat\in\Omega^{1,0}(\M)$ its metric dual (turning a vector field into a 1-form).
It is well-known that the Lie derivative of the metric along $U$ is the symmetrization of the covariant derivative of $U^\flat$. In our notations,
\beq
\calL_U \g = \dgV U^\flat + (\dgV U^\flat)^T.
\label{eq:calLg}
\eeq
In particular, for $f\in C^\infty(\M)$,
\beq
\calL_{{(df)}^\#} \g = \dgV df + (\dgV df)^T = 2\bHg f.
\label{eq:calL_df}
\eeq

\begin{lemma}
\label{lemma:2.6}
Let $U\in\VF(\M)$, let $\psi = \calL_U\g$ and let $\g(\psi,t)=\g+t\psi$. Then,
\[
\derivAtZero\Rm_{\g(\psi,t)} =\calL_U\Rm_\g.
\]
Consequently, substituting \eqref{eq:prop3.5}
\beq
\bHg\calL_U\g = 2\,\calL_U\Rm_\g +2\, D_\g\calL_U\g.
\label{eq:combine1}
\eeq
\end{lemma}

\begin{proof}
By \cite[pp.~43-44]{Pet16},
\[
\g(\psi,t)=\varphi^*_{t}\g+o(t),
\]
where $\varphi_{t}:\M\rightarrow\M$ is the flow of $U$. Since $\varphi_{t}$ is a diffeomorphism, $\varphi^*_{t}\g$ is isometric to $\g$, and by the invariance of the curvature tensor under isometries \cite[p.~199]{Lee18},
\[
\derivAtZero \Rm_{\g(\psi,t)} =
\derivAtZero \Rm_{\varphi^*_{t}\g}
=\derivAtZero \varphi^*_{t}\Rm_\g
=\calL_U\Rm_\g,
\]
where the last passage follows from the definition of the Lie derivative.
\end{proof}

\begin{lemma}
If $\g$ has constant sectional curvature $\kappa$, then for every $\psi\in\Theta^1(\M)$,
\beq
D_\g\psi = - \kappa \, \g\wedge\psi. 
\label{eq:D_for_kappa}
\eeq
\end{lemma}

\begin{proof}
Noting that
\[
i_{E_i}\Rm_\g = - \kappa\, \g\wedge (E_i^\flat)^T, 
\]
we substitute into \eqref{eq:formula_D} to obtain
\[
D_\g\psi = - \tfrac12 \kappa \, \g\wedge \sum_i( (E_i^\flat)^T \wedge (i^V_{E_i}\psi)+ (E_i^\flat)\wedge (i_{E_i}\psi)) = -\kappa\,\g\wedge\psi.
\]
Note that the last passage holds only for $k=m=1$.
\end{proof}

If we add to this the fact that by \eqref{eq:kappa5}, $\dg$ and $\dgV$ commute when acting on symmetric $(1,1)$-forms, we find that $\bHg$ operates as in \eqref{eq:explicit_H_constant}. The main result of this section is the following:

\begin{proposition}
\label{prop:2.8}
If $\g$ has constant sectional curvature, then for every $U\in\VF(\M)$,
\[
\bHg\calL_U\g = 0,
\]
and for every $f\in\Omega^{0,0}(\M)$,
\[
\bHg\bHg f = 0.
\]
\end{proposition}

\begin{proof}
If $\g$ has constant sectional curvature $\kappa$, then by the product rule for Lie derivatives,
\[
\calL_U\Rm_\g = \kappa\, \g\wedge\calL_U\g 
\EwR{eq:D_for_kappa} - D_\g \calL_U\g.
\]
Substituting into \eqref{eq:combine1} we obtain the first assertion.
The second assertion follows from the first, as for $f\in\Omega^{0,0}(\M)$,
\[
\bHg\bHg f \EwR{eq:calL_df} \tfrac12 \bHg\calL_{(df)^\#}\g = 0.
\]
\end{proof}

In fact, $\bHg\bHg f =0$ holds for a wider family of manifolds---locally-symmetric spaces:

\begin{proposition}
Let $(\M,\g)$ be a locally-symmetric space, i.e., $\nabg\Rm_\g=0$. 
For every vector field $U\in\VF(\M)$,
\[
\begin{split}
 \bHg\calL_U\g &= -\G_V dU^\flat(X;R_\g(Z,W)Y) + \G_V dU^\flat(Y;R_\g(Z,W)X) \\
&\quad + \G_V dU^\flat(W;R_\g(X,Y)Z) - \G_V dU^\flat(Z;R_\g(X,Y)W).
\end{split}
\]
In particular,
for every $f\in\Omega^{0,0}(\M)$,
\[
\bHg\bHg f = 0.
\]
\end{proposition}

\begin{proof}
We start with \eqref{eq:combine1},
\[
\bHg\calL_U\g = 2\brk{ \calL_U\Rm_\g + D_\g\calL_U \g}.
\]
Now,
\[
\begin{split}
\calL_U\Rm_\g(X,Y;Z,W) &= U(\Rm_\g(X,Y;Z,W)) - \Rm_\g(\calL_U X,Y;Z,W) - \Rm_\g(X,\calL_UY;Z,W) \\
&- \Rm_\g(X,Y;\calL_U Z,W) - \Rm_\g(X,Y;Z,\calL_U W),
\end{split}
\]
and
\[
\begin{split}
\nabg_U\Rm_\g(X,Y;Z,W) &= U(\Rm_\g(X,Y;Z,W)) - \Rm_\g(\nabg_U X,Y;Z,W) - \Rm_\g(X,\nabg_UY;Z,W) \\
&- \Rm_\g(X,Y;\nabg_U Z,W) - \Rm_\g(X,Y;Z,\nabg_U W).
\end{split}
\]
Subtracting the second equation from the first, using the fact that $\nabg\Rm_\g=0$ and the symmetry of the connection,
\[
\begin{split}
\calL_U\Rm_\g(X,Y;Z,W) &= \Rm_\g(\nabg_X U,Y;Z,W) + \Rm_\g(X,\nabg_Y U;Z,W) \\
&\quad + \Rm_\g(X,Y;\nabg_Z U,W) + \Rm_\g(X,Y;Z,\nabg_W U) \\
&= - \nabg U^\flat(X;R_\g(Z,W)Y) + \nabg U^\flat(Y;R_\g(Z,W)X) \\
&\quad + \nabg U^\flat(W;R_\g(X,Y)Z) - \nabg U^\flat(Z;R_\g(X,Y)W).
\end{split}
\]
On the other hand, by \eqref{eq:D_explicit_expression},
\[
\begin{split}
D_\g\calL_U\g(X,Y;Z,W)&= \tfrac12 \brk{\calL_U\g(X;R_\g(Z,W)Y)- \calL_U\g(Y;R_\g(Z,W)X)} \\
&+ \tfrac12 \brk{\calL_U\g(R_\g(X,Y)W;Z)- \calL_U\g(R_\g(X,Y)Z;W)}.
\end{split}
\]
Adding the two, noting that
\[
\begin{split}
\nabg U^\flat(X,Y) - \tfrac12 \calL_U\g(X,Y) &= \tfrac12 \brk{\nabg U^\flat(X,Y) -  \nabg U^\flat(X,Y)} \\
&= \tfrac12 d U^\flat(X,Y) \\
&= \tfrac12 \G_V dU^\flat(X;Y),
\end{split}
\]
we obtain the first result. The second part follows from the first, as
\[
\bHg f = \tfrac12 \calL_{(df)^\#}\g,
\]
and $\G_V d df = 0$.
\end{proof}

%
%

For $k\in\BRK{0,1}$, we further set 
\[
\bHg^*:\Omega^{k+1,k+1}(\M)\rightarrow \Omega^{k,k}(\M)
\]
to be the $L^2$-dual of $\bHg$. Denoting by $D_\g^*:\Lambda^{k+1,l+1}T^*\M\rightarrow \Lambda^{k,l}T^*\M$ the metric dual of $D_\g$, 
\[
\bHg^*\psi =\Hg^*\psi+D_\g ^*\psi.
\]
If $\g$ has constant sectional curvature $\kappa$, then by duality, 
\[
D_\g^*\psi = -\kappa\trace_\g\psi,
\]
and for every $\psi\in\Omega^{2,2}(\M)$,
\[
\bHg^*\bHg^*\psi=0.
\]


\subsection{Construction of $\bFg$}

The commutator $S_\g:\Lambda^{k,m}T^*\M\rightarrow \Lambda^{k+1,m-1}T^*\M$,
\[
S_\g = \tfrac12(\dg\deltagV-\deltagV\dg),
\]
is a tensorial operator \cite[Eq. (3.11)]{KL21a}.
We define $\bFg:\Omega^{k,m}(\M)\rightarrow\Omega^{k+1,m-1}(\M)$ by
\[
\bFg = \Fg+ S_\g = \dg\deltagV,
\]
and $\bFg^*:\Omega^{k+1,m-1}(\M)\rightarrow\Omega^{k,m}(\M)$ as its $L^2$-dual, 
\[
\bFg^* = \dgV\deltag.
\]
We note that
\[
(\bFg\lambda^T)^T = \bFg^*\lambda.
\]
In the sequel we only consider $\bFg$ acting on $\Theta^1(\M)$ and  $\bFg^*$ acting on $\Omega^{2,0}(\M)$.
We note that
\[
\tfrac12\brk{\bFg^* + (\bFg^*(\cdot))^T} : \Omega^{2,0}(\M)\to \Theta^1(\M)
\]
is $L^2$-dual to $\bFg|_{\Theta^1(\M)}$. For manifolds having constant sectional curvature, it follows from \eqref{eq:kappa5} that $S_\g\psi=0$ for $\psi\in\Theta^1(\M)$.

\begin{proposition}
\label{prop:2.9}
For $\lambda\in\Omega^{2,0}(\M)$,
\[
\bFg^*\lambda + (\bFg^*\lambda)^T=\calL_{(\delta\lambda)^{\sharp}}\g.
\]
Consequently, for $\g$ having constant sectional curvature,
\[
\bHg(\bFg^*\lambda + (\bFg^*\lambda)^T)=0.
\]
By duality,
\[
\bFg \bHg^*|_{\Theta^2(\M)} = 0.
\]
\end{proposition}

\begin{proof}
By definition,
\[
\bFg^*\lambda =\dgV\deltag\lambda = \nabg\delta\lambda,
\label{eq:F_lie_derivative}
\]
where we used the fact $\deltag=\delta$ and $\dgV=\nabg$ for scalar-valued forms. 
Hence,
\[
\bFg^*\lambda+(\bFg^*\lambda)^T =
(\nabg \delta\lambda) + (\nabg \delta\lambda)^T \EwR{eq:calLg}
\calL_{(\delta\lambda)^{\sharp}}\g.
\] 
The second assertion follows from \propref{prop:2.8}.
\end{proof}

\begin{proposition}
Let $\g$ have constant sectional curvature. Then, for every
$f\in C^{\infty}(\M)=\Omega^{0,0}(\M)$ and $\lambda\in\Omega^{2,0}(\M)$,
\[
\bFg\bHg f=0
\Textand 
\bHg^*(\bFg^*\lambda + (\bFg^*\lambda)^T) = 0.
\]
\end{proposition}

\begin{proof}
The second assertion follows from the first by duality. For $f\in C^{\infty}(\M)$,

\[
\begin{split}
\deltagV\bHg f =  \deltagV \dg \dgV f \EwR{eq:kappa4}  (\dg \deltagV - (d-1)\kappa\G)\dgV f
= d \deltagV\dgV f - (d-1)\kappa\, d f,
\end{split}
\]
where in the last passage we used the fact that $\G\dgV f = d f$. Thus, 
\[
\begin{split}
\bFg\bHg f =  \dg \deltagV \dg \dgV f = dd \deltagV\dgV f - (d-1)\kappa\, dd f = 0.
\end{split}
\]

\end{proof}

\subsection{Decomposition of symmetric forms}

We summarize the set of identities proved in the previous section:
if $(\M,\g)$ has constant sectional curvature $\kappa$, then the operators in Diagram \eqref{eq:exact_diagram} are given by
\[
\begin{gathered}
\bHg|_{\Theta^0(\M)} = \dgV d
\qquad
\bHg|_{\Theta^1(\M)} = \tfrac12(\dgV \dg + \dg\dgV)  -\kappa \,\g\wedge \\
\bHg^*|_{\Theta^1(\M)} = \deltag \deltagV
\qquad
\bHg^*|_{\Theta^2(\M)} = \tfrac12(\deltag \deltagV + \deltagV \deltag)  -\kappa \,\trace_\g \\
\bFg|_{\Theta^1(\M)} = \dg\deltagV \\
\bFg^*|_{\Omega^{2,0}(\M)} = \dgV\deltag,
\end{gathered}
\]
and satisfy
\beq
\begin{gathered}
\bHg\bHg=0 \qquad \bFg\bHg=0 \qquad \bHg(\bFg^*+(\bFg^*(\cdot))^T)=0 \\
\bHg^*\bHg^*=0 \qquad \bFg\bHg^*=0 \qquad \bHg^*(\bFg^*+(\bFg^*(\cdot))^T)=0.
\end{gathered}
\label{eq:exact_relations}
\eeq

It then follows from Theorems~7.5 and 7.9 in \cite{KL21a}:

\begin{theorem}
\label{thm:3.11}
Let $\g$ have constant sectional curvature. 
Then, the module of symmetric $(1,1)$-forms decomposes $L^2$-orthogonally into 
\[
\begin{split}
\Theta^1(\M) &= \calS\EE^1(\M)\oplus 
\calS\CC^1(\M)
\oplus \calS\EC^1(\M) \oplus \calS\BH^1(\M),
\end{split}
\]
where 
\[
\begin{aligned}
\calS\EE^1(\M) &= \{\bHg f ~:~ f\in C^\infty(\M)\cap \ker(\PttD,\frakT)\} \\
\calS\CC^1(\M) &= \{\bHg^* \psi ~:~ \psi\in\Theta^2(\M)\cap \ker(\PnnD,\frakT^*)\} \\
\calS\EC^1(\M) &= \{\bFg^* \lambda + (\bFg^*\lambda)^T ~:~ \lambda\in\Omega^{2,0}(\M)\cap \ker(\PntD,\frakF^*)\} \\
\calS\BH^1(\M) &= \Theta^1(\M)\cap \ker(\bHg,\bHg^*,\bFg). 
\end{aligned}
\] 
The biharmonic module decomposes further into either
\beq
\begin{split}
&\calS\BH^1(\M) = \calS\BH^1_{\bHg + (\bFg^* + (\bFg^*(\cdot))^T}(\M)\oplus\calS\BH_{\NN}^1(\M)
\\& \calS\BH^1(\M) = \calS\BH^1_{\bHg^* + (\bFg^* + (\bFg^*(\cdot))^T}(\M)\oplus\calS\BH_{\TT}^1(\M),
\end{split}
\label{eq:decompositon_of_symmetric_biharmonic}
\eeq
where 
\[
\begin{split}
&\calS\BH^1_{\bHg + (\bFg^* + (\bFg^*(\cdot))^T}(\M) = \calS\BH^1(\M) \cap\brk{\image\bHg + \image(\bFg^* + (\bFg^*(\cdot))^T)}, 
\\&\calS\BH^1(\M)_{\bHg^* + (\bFg^* + (\bFg^*(\cdot))^T}(\M)=\calS\BH^1(\M) \cap\brk{\image\bHg^{*} + \image(\bFg^* + (\bFg^*(\cdot))^T)}, 
\end{split}
\] 
and $\calS\BH_{\NN}^1(\M),\calS\BH_{\TT}^1(\M)$ have been defined in \eqref{eq:BHTTNN}. 
An analogous decomposition holds in Sobolev regularity $W^{s,p}$, 
for every $s\in\mathbb{N}\cup\{0\}$ and $p\geq 2$, where the potentials $f$, $\psi$ and $\lambda$ have  $W^{s+2,p}$ regularity.
All spaces in both decompositions are closed in the $W^{s,p}$-topology. 
\end{theorem}

A comment on notations: in $\calS\EE^1$, $\calS\CC^1$ and $\calS\EC^1$, the symbol $\calS$ stands for ``symmetric", whereas the symbols $\calE$ and $\calC$ stand for ``exact" and ``coexact", in analogy with the standard notation in Hodge theory \cite{Sch95b}. The symbol $\BH$ stands for ``biharmonic", which in the present context refers to the kernel of all four second-order operators. 
The submodules $\calS\BH_{\NN}^1(\M)$ and $\calS\BH_{\TT}^1(\M)$ are the kernel of the two sets of regular elliptic operators $\frakB_{\NN}$ and $\frakB_{\TT}$ defined in Section~2. As a result, they are finite dimensional and consist only of smooth sections.

\section{Saint-Venant compatibility}

By \lemref{lemma:2.6},  if $\sigma\in\Theta^1(\M)$ satisfies $\sigma=\calL_Y\g$ for some vector field $Y\in\VF(\M)$ then 
\beq
\derivAtZero \Rm_{\g(\sigma,t)} = \calL_Y\Rm_\g. 
\label{eq:SV}
\eeq
This equation is known as the Saint-Venant compatibility equation \cite{Cal61};
for $\Rm_\g=0$ it reduces to the condition that a Lie derivative of the metric is in the kernel of the map
\[
\sigma\mapsto \derivAtZero \Rm_{\g(\sigma,t)}.
\]

A natural question is whether the converse is also true: Suppose that  $\sigma\in\Theta^1(\M)$ satisfies \eqref{eq:SV} for some $Y\in\VF(\M)$; does it imply that $\sigma=\calL_X\g$ for some $X\in\VF(\M)$?
This question was answered affirmatively in the smooth category for closed, simply-connected symmetric spaces \cite{GG88}; \cite{CCGK07,GK09} provide an answer in $L^2$-regularity for simply-connected Euclidean domains,  and obtain a Hodge-like decomposition for $L^2\Theta^1(\M)$. 

The next theorem improves these results, removing both the regularity and topological assumption, as well as the assumption of a Euclidean domain.
As suggested by \propref{prop:2.8}, for spaces of constant sectional curvature, the Saint-Venant compatibility can be reformulated such that any Lie derivative of the metric lies in the kernel of $\bHg$.

\begin{theorem}
\label{thm:3.1}
Let $(\M,\g)$ be a compact Riemannian manifold with boundary having constant sectional curvature. Then, $\sigma\in\Theta^1(\M)$ satisfies
\[
\sigma = \calL_Y\g
\]
for some vector field $Y\in\VF(\M)$
if and only if
\beq
\bHg\sigma = 0
\Textand
\sigma \perp \calS\BH_\NN^1(\M).
\label{eq:SV_conditions}
\eeq
\end{theorem}

\begin{proof}
Let $\sigma$ satisfy \eqref{eq:SV_conditions}.
By \thmref{thm:3.11}, every $\sigma\in\Theta^1(\M)$ decomposes orthogonally into 
\beq
\sigma = \bHg \alpha + \bHg^*\beta +\bFg^*\lambda + (\bFg^*\lambda)^T + \kappa,
\label{eq:combine3}
\eeq
where $\alpha\in C^\infty(\M)\cap \ker(\PttD,\frakT)$, $\beta\in \Theta^2(\M)\cap \ker({\PnnD},{\frakT}^*)$, $\lambda \in \Omega^{2,0}(\M)\cap \ker({\PntD},{\frakF^*})$ and $\kappa\in \calS\BH^1(\M)$. 
By the orthogonality of the deocmposition,
\[
\bra\sigma ,\bHg^*\beta\ket = \bra\bHg^*\beta,\bHg^*\beta\ket.
\]
Integrating the left-hand side by parts using \eqref{eq:ibp} and taking into account that $\beta\in\ker(\PnnD,\frakT^*)$,
we obtain that $\bra\bHg\sigma ,\beta\ket = \bra\bHg^*\beta,\bHg^*\beta\ket$. Since $\bHg\sigma=0$,  it follows that $\bHg^*\beta=0$.

Next, since $\sigma \perp \calS\BH_\NN^1(\M)$, 
it follows from the decomposition \eqref{eq:decompositon_of_symmetric_biharmonic} of the biharmonic module that
\[
\kappa = \bHg \vp + \bFg^*\mu + (\bFg^*\mu)^T,
\]
for some $\vp\in C^\infty(\M)$ and $\mu\in \Omega^{2,0}(\M)$. Combining with \eqref{eq:combine3},
\[
\sigma = 2\, \bHg f + \bFg^*\eta + (\bFg^*\eta)^T,
\]
for some $f\in C^\infty(\M)$ and $\eta\in \Omega^{2,0}(\M)$. 
It follow from \eqref{eq:calL_df} and \propref{prop:2.9} that
\[
\sigma = \calL_Y\g
\qquad
\text{for}
\qquad
Y = (df + \delta\eta)^\#,
\]
which completes the first part of the proof.

In the other direction, let $\sigma=\calL_Y\g$. By \propref{prop:2.8},
\[
\bHg \sigma =0.
\]
As for the orthogonality to $\calS\BH^1_{\NN}(\M)$,
the Hodge decomposition for scalar 1-forms yields that $Y^\flat$ can be written in the form
\[
Y^\flat = df + \delta\lambda,
\]
for some $f\in C^{\infty}(\M)$ and $\lambda\in\Omega^{2,0}(\M)$. Using once again \eqref{eq:calL_df} and \propref{prop:2.9}, it follows that
\[
\calL_Y\g = 2\, \bHg f+\bFg^*\lambda+(\bFg^*\lambda)^T,
\]
which by the decomposition \eqref{eq:decompositon_of_symmetric_biharmonic} is orthogonal to $\calS\BH^1_{\NN}(\M)$.
\end{proof}

Since the spaces in the decompositions of \thmref{thm:3.11} are closed in any Sobolev regularity,
we may reformulate \thmref{thm:3.1} at lower regularity:

\begin{theorem}
\label{thm:kernel_H}
Let $(\M,\g)$ be a compact Riemannian manifold with boundary having constant sectional curvature. Then, for every $s\in\mathbb{N}\cup\{0\}$ and $p\geq 2$, $\sigma\in W^{s,p}\Theta^1(\M)$ satisfies
\[
\sigma = \calL_Y\g
\]
for some vector field $Y\in W^{s+1,p}\VF(\M)$
if and only if
\[
\sigma\perp\calS\CC^1(\M)
\Textand
\sigma \perp \calS\BH_\NN^1(\M).
\]
\end{theorem}

Due to the closedness of $\calS\CC^1(\M)$, the condition $\sigma\perp\calS\CC^1(\M)$ is equivalent to the distributive equation $\bHg\sigma=0$. For $s\ge2$, this equivalence is in the classical sense.

Following \cite{CCGK07,GK09}, we can restate \thmref{thm:kernel_H} in the following manner. Set
\[
\begin{gathered}
L(\M)=\{\calL_{X}\g~:~X\in\frakX(\M)\} \\
W^{s,p}L(\M)=\BRK{\calL_{X}\g~:~X\in W^{s+1,p}\frakX(\M)}.
\end{gathered}
\]
We obtain an $L^2$-orthogonal decomposition
\beq
W^{s,p}\Theta^1(\M)=W^{s,p}\calS\CC^1(\M)\oplus W^{s,p}L(\M)\oplus\calS\BH^1_{\NN}(\M),
\label{eq:saint_venant_decompositon}
\eeq
where the splitting is both algebraic and topological, and each space is closed in the $W^{s,p}$-topology.
 
Equation \eqref{eq:saint_venant_decompositon} is a direct generalization of the Hodge-like decomposition in \cite{GK09} for Euclidean domains. In their notation, with $(\M,\g)$ a Euclidean domain, $\calS\BH^1_\NN(\M)=\bbK$, where
\[
\bbK = \Theta^1(\M) \cap \ker{(\deltag,\bHg,\PnD)}
\]
(see \corrref{cor:1_symmetric_biharmonic_fields} for the equivalence). In their version of the decomposition, \cite{GK09} points out the ability to decompose $L(\M)$ further to include an element whose potential vanishes on the boundary. Although this is not apparent at first glance, it is worth pointing out that our result generalizes this as well. Moreover, our decomposition differentiates whether the potential of the Lie derivative is a gradient field or the codifferential of a 1-form:

\begin{proposition}
In any regularity, $L(\M)$ further splits  into
\[
L(\M)=\S\EE^1(\M)\oplus \S\CE^1(\M)\oplus\S\BH^1_{\bHg+(\bFg^{*}+(\bFg^{*})^T)}(\M).
\]
Moreover,
\[
\S\EE^1(\M),\ \S\CE^1(\M)\subseteq\BRK{\calL_{X}\g ~:~X\in\frakX(\M),\ X|_{\dM}=0}.
\]
\end{proposition}

\begin{proof}
The first part follows from a direct comparison of \eqref{eq:saint_venant_decompositon} and \thmref{thm:3.11}. As for the second part, let 
\[
\bHg f =\half\calL_{{(df)}^{\sharp}}\g\in \S\EE^1(\M),
\qquad
f\in\ker(\PttD,\frakT).
\]
From \eqref{eq:1st_order_proj}, this amounts to $\PttD f=0$ and $\PntD df=0$, hence $df|_{\dM}=0$. 

Next, let
\[
\bFg^{*}\lambda+(\bFg^{*}\lambda)^T=\calL_{(\delta\lambda)^{\sharp}}\g\in\S\CE^1(\M)
\qquad
\lambda\in\ker{(\PntD,\frakF^{*})}. 
\]
We need to show that $\PntD\delta\lambda=0$ and $\PttD\delta\lambda=0$.
By \cite[Eq. 4.18d]{KL21a}, using the fact that $\PnnD\lambda=0$,
\[
\PntD\delta\lambda=-\delta\PntD\lambda=0.
\]
From \eqref{eq:1st_order_proj}, using once again the fact that  $\PnnD\lambda=0$, we find that $\frakF^{*}\lambda=0$ amounts to,
\[
\PttD\delta\lambda=\PnnD\dgV\lambda-\deltagD\PttD\lambda.
\]
From the commutation formulas in \cite[Lemmas~4.9-4.10]{KL21a} and the fact that $\PntD\lambda=0,\PnnD\lambda=0$, 
\[
\PnnD\dgV\lambda=\PntD\nabg_{\frakn}\lambda
\Textand 
\deltagD\PttD\lambda=\PttD\delta\lambda + \PntD\nabg_{\frakn}\lambda,
\]
which yields that $\PttD\delta\lambda=0$. 
\end{proof}

If $\M$ is simply-connected and locally-flat, then the orthogonality to the biharmonic module holds trivially, as:

\begin{proposition}
\label{prop:locally_flat_top}
Let $(\M,\g)$ be simply-connected and locally-flat. Then, $\sigma\in\Theta^1(\M)$ satisfies 
\[
\sigma = \calL_Y\g
\]
for some vector field $Y\in\VF(\M)$ if and only if 
\[
\bHg\sigma = 0.
\]
Comparing with \thmref{thm:3.1},
\[
\calS\BH^1_\NN(\M) = \{0\}.
\]
\end{proposition}

\begin{proof}
The ``only if" follows from \propref{prop:2.8}. As for the ``if" part, note first that every simply-connected locally-flat manifold is isometric to a Euclidean domain. Since
\[
\bHg\sigma = \dg\dgV\sigma = 0,
\]
it follows from the de-Rham cohomology that there exists an $A\in\Omega^{0,2}(\M)$, such that
\[
\dgV\sigma = \dg A.
\]
Since $\sigma$ is symmetric, $\G_V\sigma=0$, and since $\G_V$ anti-commutes with $\dgV$,
\[
0 = \dgV\G_V\sigma = -\G_V\dg A.
\]
By \cite[Lemma.~3.7]{KL21a}, and since $\G_V$ annihilates every $(0,m)$-form,
\[
\G_V\dg A = \dg\frakG_V A+\frakG_V\dg A = \dgV A,
\]
from which follows that
\[
\dgV A = 0.
\] 
By the de-Rham cohomology, there exists an $\eta\in\Omega^{0,1}(\M)$, such that $A = \dgV \eta$, namely,
\[
0 = \dgV\sigma - \dg A = \dgV(\sigma - \dg \eta).
\]
Once more application of the de-Rham cohomology implies the existence of $\omega\in\Omega^{1,0}(\M)$, such that
\[
\sigma = \dg\eta + \dgV\omega.
\] 
Since $\sigma$ is symmetric, we may symmetrize the right-hand side, yielding
\[
\sigma = \calL_Y\g,
\]
for
\[
Y^\flat = \tfrac12(\omega + \eta^T). 
\]
\end{proof}

The natural question is whether $\calS\BH^1_\NN(\M)=\BRK{0}$ for $\M$ simply-connected and $\g$ having constant sectional curvature.
Calabi \cite{Cal61} builds upon this very proof of the locally-flat case to prove that in manifolds with constant curvature without boundary, $\bHg\sigma=0$ implies $\sigma\in L(\M)$. The following theorem generalizes Calabi's result to manifolds with boundaries, albeit only for positive sectional curvature, using a different technique:

\begin{theorem}
If $(\M,\g)$ is simply-connected and has positive constant sectional curvature, then $\calS\BH^1_{\NN}(\M)=\BRK{0}$.
\end{theorem}
\begin{proof}
Since $(\M,\g)$ is simply-connected, it can be isometrically embedded in a closed sphere of radius $R$, which in turn can be isometrically embedded as a hypersurface in euclidean space.  
This setting can be realized as an isometric embedding $j:(\M,\g)\hookrightarrow(\Brk{0,1}\times \M,\frake)$ with a smooth distance function $r:I\times \M\rightarrow\mathbb{R}_{\geq 0}$
and a Euclidean metric $\frake$ of the form
\beq
\frake = dr\otimes dr+\bar{\g}.
\label{eq:exp_euc}
\eeq
The tensor $\bar{\g}\in \Theta^1(I\times \M)$ has only tangent parts, and on each level set of $r$ restricts to its intrinsic metric. The level sets of $r$ are spheres as well. Thus, the Riemannian metrics of the level sets are conformal to $(\M,\g)$,
\[
\bar{\g}(r,x)= c(r) \,\g(x),
\qquad
\text{where}\qquad
c(r)=\frac{(R+r)^2}{R^2}.  
\]
The second fundamental form of these level sets, $\bar{\frakh}\in\Theta^1(U)$ is
\[
\bar{\frakh}_{ij}(r,x)=\frac{R+r}{R^2}\,\frakg_{ij}(x),
\]
and the corresponding shape operator is
\[
\bar{S}=\frac{1}{R+r}\id.
\]
A direct calculation shows that,
\[
\begin{gathered}
\frac{c'(r)}{c(r)}=\frac{2}{R+r}
\qquad\text{hence}
\qquad
\frac{1}{R^2}-\frac{c'(r)}{c(r)}\frac{R+r}{R^2}=\kappa,
\end{gathered}
\]
where $\kappa$ is the constant curvature of $(\M,\g)$ (which in our convention is $-1/R^2$). 

Let $\sigma\in\calS\BH^1_{\NN}(\M)$ satisfy $\Hg\sigma=0$, and consider $\bar{\sigma}\in\Theta^1(I\times \M)$ given in semi-geodesic coordinates by
\[
\bar{\sigma}(r,x)=c(r)\,\sigma(x).
\]
It follows that $\bar{\sigma}$ has no normal components and
\[
\calL_{\dr}\bar{\sigma}=c'(r)\,\bar{\sigma}. 
\]

For $\e>0$, we denote by 
\[
\Ptt : \Theta^1(I\times\M) \to \Theta^1(\M)
\]
the pullback of a double form onto the level set $r^{-1}(\{\e\})$, which can be identified with $\M$. Similarly we define the boundary projections operator $\Ptn$ and $\Pnn$.
Let $\gEps = \Ptt\euc$ denote the pullback metric of $r^{-1}(\{\e\})$; as seen above, $\gEps=c(\e)\g$. 
Using the fact that the connection is invariant under constant conformal factors \cite[pp.~217]{Lee18} and that $\dg$ is determined  by the connection of $\g$, we find 
\[
H_{\gEps}=\Hg 
\qquad \dgEps=\dg 
\Textand \dgEpsV=\dgV.
\]  

A direct calculation using the commutations relations derived in \cite[Section~4]{KL21a} gives that $\bHg\sigma=0$ implies that $\Ptt H_\euc\bar{\sigma}=0$,
\[
\begin{split}
\Ptt H_{\frake}\bar{\sigma}&=H_{\gEps}\Ptt\bar{\sigma}+\frakh_{\e}\wedge\frakT_{\e}\bar{\sigma}
\\&=H_\g c(\e)\sigma+\frac{R+\e}{R^2}\g\wedge\brk{c'(\e)-\frac{c(\e)}{R+\e}}\sigma
\\&=c(\e)\brk{\Hg\sigma-\brk{\frac{1}{R^2} - \frac{c'(\e)}{c(\e)}\frac{R+\e}{R^2}}\,\g\wedge\sigma}
\\&=c(\e)\brk{H_\g\sigma-\kappa\,\g\wedge\sigma}
\\&=c(\e)\bHg\sigma=0. 
\end{split}
\]
$\Ptn H_\euc\bar{\sigma}=0$ and $\Pnn H_\euc\bar{\sigma}=0$ are proven in a similar fashion, which implies that
\[
H_\euc\bar{\sigma}=0.
\]
Since $I\times\M$ is a simply-connected flat space, we conclude from \propref{prop:locally_flat_top} that $\bar{\sigma}=\calL_{\bar{X}}\frake$ for some vector field $\bar{X}\in\frakX(I\times \M)$. Decompositing it into tangent and normal parts,
\[
\bar{X}=\bar{X}^{\parallel}+ \bar{X}^{r}\dr,
\]
and inserting expression \eqref{eq:exp_euc} for $\frake$, keeping in mind that $i_{\dr}\bar{\g}=0$, 
\[
\begin{split}
\bar{\sigma}&=\calL_{\bar{X}}(dr\otimes dr+\bar{\g})
\\&=\calL_{\bar{X}}dr\otimes dr+dr\otimes \calL_{\bar{X}}dr+\calL_{\bar{X}^{\parallel}+\bar{X}^{r}\dr}\bar{\g}
\\&=\calL_{\bar{X}}dr\otimes dr+dr\otimes \calL_{\bar{X}}dr+\calL_{\bar{X}^{\parallel}}\bar{\g}+\bar{X}^{r}\calL_{\dr}\bar{\g}. 
\end{split}
\]
$\calL_{\bar{X}^{\parallel}}\bar{\g}$ has no normal parts, since both $\bar{X}^{\parallel}$ and $\bar{\g}$ have no normal parts. Moreover, $\calL_{\dr}\bar{\g}=2\,\bar{\frakh}$, which also has no normal parts. Thus,  since $\bar{\sigma}$ has no normal components, we conclude that $\calL_{\bar{X}}dr=0$, i.e, $\bar{X}^{r}=const$. Therefore,
\[
\bar{\sigma}=\calL_{X^{\parallel}}\bar{\g}+2\,\bar{X}^{r}\bar{\frakh}
\]
Restricting to $\M$, setting $X=X^{\parallel}|_{\M}$ and imposing $\bHg\sigma=0$ yields that $\bar{X}^{r}=0$, hence
\[
\sigma=\calL_{X}\g.
\]
Comparing once again with \thmref{thm:3.1}, $\calS\BH^1_{\NN}(\M)=\BRK{0}$.
\end{proof}

\section{Equations of incompatible elasticity}

\subsection{Incompatible elasticity}
\label{sec:application}


Let $(\M,\g)$ and $(\bar{\M},\bar{\g})$ be smooth $d$-dimensional Riemannian manifolds. 
The manifold $\M$ is compact with a boundary, and it represents the body; 
the manifold $(\bar{\M}$ has no boundary and it represents space. 
For a configuration $f:\M\to\bar{\M}$, let
\[
W:T^*\M\otimes \fTN \to \R
\]
be an elastic energy density,  usually assumed to possess  symmetries. The stored energy associated with a configuration $f:\M\to\bar{\M}$ (in the absence of body forces or boundary constraints) is
\beq
\calE(f) = \int_\M W(df)\, \VolumeG.
\label{eq:energy}
\eeq
The Euler-Lagrange equations for the critical points of \eqref{eq:energy} can be formulated in terms of double forms. The resulting boundary-value problem for the stress field, $\sigma\in\Theta^1(\M)$ are  
\beq
\Cases{
\delfh \sigma = 0 & \text{in $\M$} \\
\Pnfh\sigma = 0 & \text{on $\partial\M$}.
}
\label{eq:equilbrium_sigma}
\eeq
where $\fh$ is the pullback metric on $\M$ induced by the critical point $f$ of the energy functional \eqref{eq:energy}. 

Equations \eqref{eq:equilbrium_sigma} form (in local coordinates) a system of $d$ differential equations for the $d(d+1)/2$ components of the stress $\sigma$, which are supplemented by algebraic constitutive relations (fiber derivatives of $W$), relating $\sigma$ to the metric $\fh$. 

If the space manifold $(\bar{\M},\bar{\g})$ is Euclidean, then a simply-connected Riemannian manifold
$(\M,G)$ can be isometrically-immersed in $(\bar{\M},\bar{\g})$ if and only if $\calR_G= 0$, where $\calR_{G}\in\End(\Lambda^2T\M)$ is the curvature operator of $G$. 
Together with this compatibility condition,
Eq.~\eqref{eq:equilbrium_sigma} forms a closed boundary-value problem for the stress $\sigma$, 
or equivalently, 
for the pullback metric $\fh$. Note how nonlinear the system is: in addition to the ``constitutive nonlinearity", i.e., the nonlinear relation between $\fh$ and $\sigma$, there is also a ``geometric nonlinearity", as the differential operators depend on the unknown pullback metric $\fh$.

We rewrite the system \eqref{eq:equilbrium_sigma} in the following form,
\beq
\delta^{\nabla^G}\sigma = 0
\qquad
\qquad
\bbP^\frakn_G \sigma = 0,
\label{eq:summary1}
\eeq
where $G\in\Theta^1(\M)$ is a metric related to the stress field $\sigma\in\Theta^1(\M)$ by a nonlinear, invertible strain-stress constitutive relation $\calA:\Theta^1(\M)\rightarrow\Theta^1(\M)$,
\beq
\sigma = \calA G
\label{eq:summary2}
\eeq
satisfying $\calA \g=0$ (as the reference metric of the body, $\g$, is its strain-free state).
This system is supplemented by the compatibility condition
\beq
\calR_G = \calR_{\calA^{-1}\sigma} = 0.
\label{eq:summary4}
\eeq

Eqs. \eqref{eq:summary1}--\eqref{eq:summary4} form a closed nonlinear system of equations for $\sigma$. In the small strain limit, linearizing about $\sigma=0$ (and correspondingly $G=\g$), \eqref{eq:summary1} reduces to
\beq
\delta^{\nabg}\sigma = 0
\qquad
\qquad
\bbP^\frakn_\g \sigma = 0.
\label{eq:summary5}
\eeq
The linearization of \eqref{eq:summary4} yields
\beq
\bHg\sigma = \calC\Rm_\g, 
\label{eq:summary6}
\eeq
where $\calC$ is a linear operator related to the linearization of the constitutive relation, namely to $d\calA_\g$. 
Eqs.~\eqref{eq:summary5},\eqref{eq:summary6} are of the form of the linearized stress equations described in the introduction. From \eqref{eq:H_on_the_metric}, $\Rm_\g\in\image{\bHg}$.

\subsection{Existence and uniqueness of solutions}

Linearized incompatible elasticity gives rise to boundary-value problems for $\sigma\in\Theta^1(\M)$ of the form
\beq
\deltag\sigma = 0
\qquad  
\bHg\sigma=\calR 
\qquad
(\PnnD,\PtnD)\sigma=(\rho,\tau),
\label{eq:Euler_lagrange_equations}
\eeq
where $\calR\in\Theta^2(\M)\cap\ker\frakG\cap\image{\bHg}$ is an algebraic curvature, and $\tau\in\Omega^{1,0}(\dM)$ and $\rho\in C^{\infty}(\dM)$ are the components of the boundary traction. 
In multiply-connected manifolds, the local compatibility condition on $\bHg\sigma$ is often supplemented with a non-local constraint for each generator of the fundamental group of $\M$ \cite{KMS15}.

\begin{lemma}
\label{lem:boundary_data_delta}
Let $\sigma\in\Theta^1(\M)$, $\rho\in C^{\infty}(\dM)$ and $\tau\in\Omega^{1,0}(\dM)$. Suppose that
\[
(\PnnD,\PtnD)\sigma=(\rho,\tau).
\]
Then,
\[
(\frakT^*,\frakF)\sigma=(-\delta\tau,-d\rho + \tfrac12 \trace_{\gD}(\frakh_{0}\wedge\tau))
\]
if and only if
\[
\deltag\sigma|_{\dM}=0.
\]
Here $\frakh_{0}\in\Theta^1(\dM)$ is the scalar second fundamental form of the boundary \cite[Sec.~4.1]{KL21a}, and $d$ and $\delta$ are the  exterior derivative and its dual at the boundary.
\end{lemma}

\begin{proof}
By the definition \eqref{eq:1st_order_proj} of $\frakT^*$, since $\sigma^T=\sigma$,
\[
\frakT^*\sigma= - \PtnD\deltag\sigma- \deltagD\PtnD\sigma = - \PtnD\deltag\sigma-\delta\tau,
\]
where we used the fact that $\deltagD=\delta$ for scalar forms. Thus, $\PtnD\deltag\sigma=0$ if and only if $\frakT^*\sigma=-\delta\tau$. On the other hand, by the definition \eqref{eq:1st_order_proj} of $\frakF$ combined with the commutation relations between $\deltagV$ and $\PttD$, and $\dg$ and $\PnnD$ (Eqs. (4.17),(4.18) in \cite{KL21a}),
\[
\begin{split}
\frakF\sigma &= \tfrac12(\PnnD\dg\sigma-\dgD\PnnD\sigma)-\tfrac12(\deltagDV\PttD\sigma+\PttD\deltagV\sigma) \\
&=  \tfrac12(-2 \dgD\PnnD\sigma - \calS_0\PtnD\sigma)-
\tfrac12(2\PttD\deltagV\sigma - (\calS^*_0(\PtnD\sigma)^T)^T) \\
&=  - \PttD\deltagV\sigma  - d\rho - 
\tfrac12 \calS_0\tau + \tfrac12 (\calS^*_0\tau^T)^T.
\end{split}
\]
By the definitions of $\calS_0$, $\calS_0^*$, for $\tau\in\Omega^{1,0}(\dM)$ (Lemma 4.5 in  \cite{KL21a}),
\[
\calS_0\tau(X) = \tau(S_0(X))
\Textand
(\calS^*_0\tau^T)^T(X) = (\trace_\g \h_0)\,\tau(X),
\]
from which follows that
\[
(\calS^*_0\tau^T)^T - \calS_0\tau = \trace_{\gD}(\h_0\wedge\tau).
\]
Thus, $\PttD\deltagV\sigma=0$ (and equivalently, $\PttD\deltag\sigma=0$) if and only if $\frakF\sigma=-d\rho+\tfrac12 \trace_{\gD}(\frakh_{0}\wedge\tau)$. Finally $\PttD\deltag\sigma=0$ and $\PtnD\deltag \sigma=0$ are equivalent to $\deltag\sigma|_{\dM}=0$. 
\end{proof}

\begin{corollary}
\label{cor:1_symmetric_biharmonic_fields}
Every $\nu\in\calS\BH^1_{\NN}(\M)$ solves the homogeneous version of \eqref{eq:Euler_lagrange_equations}, i.e., 
\[
\deltag\nu = 0
\qquad  
\bHg\nu= 0 
\qquad
(\PnnD,\PtnD)\nu=(0,0).
\]
Thus,
\[
\ker(\bHg,\bHg^*,\bFg,\PnnD,\PtnD,\frakT^*,\frakF) = \ker(\deltag,\bHg,\PnnD,\PtnD).
\]
\end{corollary}

\begin{proof}
Let $\nu\in\calS\BH^1_{\NN}(\M)$, i.e., 
\[
\nu\in\ker(\bHg,\bHg^*,\bFg,\PnnD,\PtnD,\frakT^*,\frakF).
\]
The second and third assertions are the automatically satisfied. It remains to show that $\deltag\nu=0$. It follows from \lemref{lem:boundary_data_delta} (with $\rho=0$ and $\tau=0$) that
\[
\deltagV\nu|_{\partial\M}=0.
\]
The fact that $\bFg\nu=0$ and $\bHg^*\nu=0$ amounts to
\[
d\deltagV\nu=0 \Textand \delta\deltagV\nu=0,
\]
hence, $\deltagV\nu\in\Omega^{1,0}(\M)$ is a harmonic 1-form satisfying vanishing boundary conditions. From the uniqueness of the solution to the Dirichlet problem for the Hodge Laplacian \cite[Thm.~3.4.10]{Sch95b},
\[
\deltagV\nu=0.
\]
Since $\nu$ is symmetric, the same holds for $\deltag\nu\in\Omega^{0,1}(\M)$. 
\end{proof}

\begin{proposition}
\label{prop:4.2}
Let $\sigma\in\Theta^1(\M)$ satisfy 
\[
\deltag\sigma=0
\qquad
(\PnnD,\PtnD)\sigma = (\rho,\tau).
\]  
Then, 
\[
\begin{gathered}
\sigma\in\ker{(\bHg^*,\bFg)} \\
\frakT^*\sigma=-\delta\tau
\Textand
\frakF\sigma = -d\rho + \tfrac12\trace_{\gD}(\frakh_{0}\wedge\tau).
\end{gathered}
\]
\end{proposition}

\begin{proof}
Since $\deltag\sigma=0$, and in particular on the boundary, the expressions for $\frakT^*\sigma$ and $\frakF\sigma$
follow from \lemref{lem:boundary_data_delta}.
By the symmetry of $\sigma$, $\deltagV\sigma=0$, hence 
\[
\bHg^*\sigma=\delta\deltagV\sigma=0
\Textand
\bFg\sigma = \dg\deltagV\sigma=0.
\]
\end{proof}

\begin{corollary}
Let $\sigma$ be a solution of the boundary-value problem 
\eqref{eq:Euler_lagrange_equations}, then it is a solution to the regular elliptic system
\beq
\begin{gathered}
\Bg\sigma=\bHg^*\calR
\\(\PnnD,\PtnD,\frakT^*,\frakF)\sigma=(\rho,\tau,-\delta\tau,-d\rho+\tfrac12\trace_{\gD}(\frakh_{0}\wedge\tau))
\\(\PnnD\bHg,\frakT^*\bHg)\sigma= (\PnnD,\frakT^*)\calR.
\end{gathered}
\label{eq:bilaplcian_equations_stress}
\eeq
The conditions on $(\frakT^*,\frakF)\sigma$ are equivalent to $\deltag\sigma|_{\dM}=0$.  
\end{corollary}

\begin{proof}
Since by \propref{prop:4.2}, $\sigma\in\ker(\bHg^*,\bFg)$, it follows that
\[
\Bg\sigma=\bHg^*\bHg\sigma+\bHg\bHg^*\sigma+\bFg^*\bFg\sigma+(\bFg^*\bFg\sigma)^T = \bHg^*\calR.
\]
The boundary conditions follows from \propref{prop:4.2}.
\end{proof}

By \cite[Thm.~6.4]{KL21a}, the system \eqref{eq:bilaplcian_equations_stress} has a solution  $\sigma\in W^{s+2,p}\Theta^1(\M)$ for every choice of 
$\calR\in W^{s,p}\Theta^2(\M)$ and $\tau,\rho\in W^{s+2-1/p,p}\Omega^{*,*}(\dM)$.
Moreover, the solution is unique up to an element in the finite-dimensional biharmonic module $\calS\BH_{\NN}^1(\M)$. 
Thus, \emph{if} the boundary-value problem \eqref{eq:Euler_lagrange_equations} is solvable, then its solution must coincide with \emph{a} solution of the regular elliptic system \eqref{eq:bilaplcian_equations_stress}. In fact, due to \corrref{cor:1_symmetric_biharmonic_fields}, if \eqref{eq:Euler_lagrange_equations} is solvable, then \emph{every} solution of \eqref{eq:bilaplcian_equations_stress} solves \eqref{eq:Euler_lagrange_equations} as well.
Thus, the natural question is under what conditions on $\calR$, $\rho$ and $\tau$, is \eqref{eq:Euler_lagrange_equations} solvable.

Note that we have made no assumption about neither the geometry nor the topology of $(\M,\g)$. Thus, \eqref{eq:bilaplcian_equations_stress} constitutes a generalization of the biharmonic equations for the stress/strain field in classical elasticity (see e.g., \cite[p.~133]{Gur72}, where $S$ is the stress field and $E$ is the strain) to the setting of incompatible elasticity; our equation is supplemented by a complete set of boundary conditions and a uniqueness clause. 

Under the assumption of constant sectional curvature, the decomposition in \thmref{thm:3.11} enables us to take the solution to the boundary-value problem \eqref{eq:Euler_lagrange_equations} one step further. 

We first need the following lemma:

\begin{lemma}
\label{lem:delta_Hg_star}
Let $\g$ have constant sectional curvature.
For all $\psi\in\Theta^2(\M)$
\[
\bHg^*\psi\in\ker{(\deltag,\deltagV)}.
\]
\end{lemma}

\begin{proof}
Since $\bHg^*\psi\in\Theta^1(\M)$, it suffices to prove that $\deltagV\bHg^*\psi=0$. Let $\omega\in\Omega^{1,0}(\M)$ be compactly supported, then 
\[
\begin{split}
\bra\deltagV\bHg^*\psi,\omega\ket 
&= \bra\bHg^*\psi,\dgV\omega\ket \\
&= \tfrac12 \bra\bHg^*\psi,\dgV\omega + (\dgV\omega)^T \ket \\
&=  \tfrac12\bra\bHg^*\psi,\calL_{\omega^{\sharp}}\g\ket \\
&= \tfrac12\bra\psi,\bHg\calL_{\omega^{\sharp}}\g\ket \\
&= 0.
\end{split}
\]
The passages to both the first and fourth lines follow from integration by parts;
the passage to the second line follows from the symmetry of $\bHg^*\psi$; the passage to the third line follows from expression \eqref{eq:calLg} for the Lie derivative of the metric; finally, the passage to the fifth line follows from \propref{prop:2.8}, which holds for spaces of constant sectional curvature. Since this identity holds for arbitrary compactly supported $\omega\in\Omega^{1,0}(\M)$, it follows that $\deltagV\bHg^*\psi=0$. 
\end{proof}

This brings us to the main theorem of this section:

\begin{theorem}
\label{thm:5.4}
Consider the space of smooth Killing 1-forms,
\[
K(\M)=\BRK{\omega\in\Omega^{1,0}(\M)~:~\calL_{\omega^{\sharp}}\g=0}.
\]
Under the assumption that $(\M,\g)$ has constant sectional curvature, there exists a solution $\sigma\in\Theta^1(\M)$ to the boundary-value problem \eqref{eq:Euler_lagrange_equations} if and only if
\beq
\calR\in\image\bHg,
\label{eq:integrability_conditionsR}
\eeq
and
\beq
\int_{\dM}\Brk{(\rho,\PntD\omega)_{\gD}+(\tau,\PttD\omega)_{\gD}}\VolumeD=0
\qquad
\forall\omega\in K(\M).
\label{eq:integrability_conditions}
\eeq
The solution is unique up to an element $\nu\in\calS\BH^1_{\NN}(\M)$. 
In particular, there exists a unique solution orthogonal to $\calS\BH^1_{\NN}(\M)$.
If $\calR$, $\rho$ and $\tau$ are Sobolev sections, then the solution inherits the regularity, with for all $1\leq q\leq p$,
\beq
\Norm{\sigma_\nu}_{W^{s+2,p}(\M)}\lesssim \Norm{\calR}_{W^{s,p}(\M)}+\Norm{\tau}_{W^{s+2-1/p,p}(\dM)}+\Norm{\rho}_{W^{s+2-1/p,p}(\dM)}+\Norm{\nu}_{L^{q}(\M)},
\label{eq:sobolev_estimate}
\eeq
where fractional Sobolev spaces on manifolds are defined e.g. in \cite[Chapter~13]{Tay11c}.
\end{theorem}

\begin{proof}
Let $\sigma\in\Theta^1(\M)$ be a solution to \eqref{eq:Euler_lagrange_equations}. Obviously, $\calR\in\image{\bHg}$. 
Since $\deltagV\sigma=0$, it follows from integration by parts (Eq. (4.23) in \cite{KL21a}) that for all $\omega\in\Omega^{1,0}(\M)$,
\[
0 =\bra\deltagV\sigma,\omega\ket =
\bra\sigma,\dgV\omega\ket +
\int_{\dM}\Brk{(\rho,\PntD\omega)_{\gD}+(\tau,\PttD\omega)_{\gD}}\VolumeD.
\]
As in the proof of \lemref{lem:delta_Hg_star}, since $\sigma$ is symmetric, 
\[
\bra\sigma,\dgV\omega\ket = 
\tfrac12\bra\sigma,\dgV\omega+(\dgV\omega)^T\ket =
\tfrac12\bra\sigma,\calL_{\omega^{\sharp}}\g\ket.
\]
By the definition of Killing 1-forms, $\calL_{\omega^{\sharp}}\g = 0$ for all 
$\omega\in K(\M)$, proving the necessity of the compatibility condition \eqref{eq:integrability_conditions}.
As for the uniqueness clause, let $\sigma,\sigma'\in\Theta^1(\M)$ be solutions of \eqref{eq:Euler_lagrange_equations}, then by \propref{prop:4.2},
\[
\sigma-\sigma'\in\ker{(\bHg,\bHg^*,\bFg,\PnnD,\PtnD,\frakT^*,\frakF)}=\calS\BH^1_{\NN}.
\]

We proceed to prove the sufficiency of conditions \eqref{eq:integrability_conditionsR} and \eqref{eq:integrability_conditions}. 
We first argue that we may take $\calR=0$. Let $\calR\in\image\bHg$. Then $\calR=\bHg\psi$ for some $\psi\in\Theta^1(\M)$; decomposing $\psi$ according to \thmref{thm:3.11},
\[
\psi = \bHg\alpha + \bHg^*\beta + \bFg^*\lambda + (\bFg^*\lambda)^T + \kappa,
\]
where 
\[
\beta\in\ker(\PnnD,\frakT^*),
\]
since $\bHg\bHg=0$, $\bHg\bFg^{*}=0$ and $\bHg(\calS\BH^1(\M))=0$, we find that
\[
\calR = \bHg\psi = \bHg\bHg^*\beta. 
\]
By \lemref{lem:delta_Hg_star}, $\deltag \bHg^*\beta=0$, and by \cite[Cor.~7.2]{KL21a},
\[
(\PnnD,\PtnD)\bHg^*\beta = (0,0).
\] 
Thus, $\bHg^*\beta$ is a solution to \eqref{eq:Euler_lagrange_equations} for vanishing boundary data.
It follows that $\sigma$ is a solution of \eqref{eq:Euler_lagrange_equations} if and only if $\sigma - \bHg^*\beta$ solves the same equation with $\calR=0$.
 
Let $\sigma\in\Theta^1(\M)$ be a solution of the regular elliptic system \eqref{eq:bilaplcian_equations_stress} with $\calR=0$, \[
\Bg\sigma=\bHg\bHg^*\sigma+\bHg^*\bHg\sigma +\bFg^*\bFg\sigma+(\bFg^*\bFg\sigma)^T =0,
\]
and boundary data
\[
\begin{gathered}
(\PnnD,\PtnD,\frakT^*,\frakF)\sigma=(\rho,\tau,-\delta\tau,-d\rho+\tfrac12\trace_{\gD}(\frakh_{0}\wedge\tau)) \\
(\PnnD\bHg,\frakT^*\bHg)\sigma=(0,0).
\end{gathered}
\]
Thus, $\bHg^*\bHg\sigma\in\calS\CC^1(\M)$, which is orthogonal to the images of $\bHg$ and $\bFg$ by integration by parts \cite[Prop.~7.1]{KL21a}, hence the biharmonic equation splits into
\[
\bHg^*\bHg\sigma=0 
\Textand  
\bHg\bHg^*\sigma + \bFg^*\bFg\sigma + (\bFg^*\bFg\sigma)^T=0.
\]
Using once again the boundary data for $\bHg\sigma$,
\[
\bra\bHg\sigma,\bHg\sigma\ket = \bra\bHg^*\bHg\sigma,\sigma\ket =  0, 
\]
from which we conclude that
\[
\bHg\sigma=0.
\]

By the same argument as in the proof of \thmref{thm:3.1},
\[
0 = \bHg\bHg^*\sigma + \bFg^*\bFg\sigma + (\bFg^*\bFg\sigma)^T = \calL_{\omega^{\sharp}}\g,
\]
where $\omega\in\Omega^{1,0}(\M)$ is given by
\[
\omega= \tfrac12 d\bHg^*\sigma+\delta \bFg\sigma.
\]
Thus, $\omega\in K(\M)$. By the integrability condition \eqref{eq:integrability_conditions}, reversing the calculation at the beginning of the proof,
\[
0 = \int_{\dM}\Brk{(\rho,\PntD\omega)_{\gD}+(\tau,\PttD\omega)_{\gD}}\VolumeD 
=\bra\deltagV\sigma,\omega\ket - \bra\sigma,\dgV\omega\ket,
\]
however once again,
\[
\bra\sigma,\dgV\omega\ket = \bra\sigma,\dgV\omega + (\dgV\omega)^T \ket = \bra\sigma,\calL_{\omega^{\sharp}}\g\ket = 0,
\]
hence $\bra\deltagV\sigma,\omega\ket=0$.
Substituting back the definition of $\omega$,
\[
\tfrac12 \bra \deltagV\sigma, d\bHg^*\sigma\ket +\bra\deltagV\sigma,\delta\bFg\sigma \ket = 0.
\]
By \lemref{lem:boundary_data_delta}, given the boundary data, $\deltagV\sigma$ vanishes at the boundary, hence integrating by parts,
\[
\tfrac12\bra\bHg^*\sigma,\bHg^*\sigma\ket  +
\bra\bFg\sigma,\bFg\sigma\ket   =0,
\]
from which we conclude that $\bHg^*\sigma=0$ and $\bFg\sigma=0$. 

We have just established that
\[
\begin{gathered}
d\deltagV\sigma=0 \qquad \delta\deltagV\sigma=0
\Textand
\deltagV\sigma|_{\dM}=0.
\end{gathered}
\] 
Hence, $\deltagV\sigma\in\Omega^{1,0}(\M)$ is a harmonic 1-form satisfying vanishing boundary conditions, from which we conclude that
\[
\deltagV\sigma=0,
\]
that is, $\sigma$ is a solution of the boundary-value problem \eqref{eq:Euler_lagrange_equations}.

If the  data are Sobolev sections, the arguments remains the same, as the solution to the regular elliptic problem inherits the regularity of the data \cite[Prop.~7.4, Thm.~6.4]{KL21a} (the former is required to estimate $\bHg^*\beta$). 
Note that by Korn's inequality \cite[Ch.~5.12]{Tay11a} (weak) Killing fields are smooth, so $\omega\in K(\M)$ even in the non-smooth case.

\end{proof}

In fact, as is well-known, the space $K(\M)$, and hence the obstruction \eqref{eq:integrability_conditions}, is finite-dimensional.
The solvability conditions \eqref{eq:integrability_conditionsR} and \eqref{eq:integrability_conditions} are in general not easy to verify. There are however situations of practical interest in which they can be shown to hold.
As explained in \secref{sec:application}, in (linearized) incompatible elasticity, the source term $\calR$ is the Riemannian curvature tensor, which can be expressed as $\bHg\g$, hence \eqref{eq:integrability_conditionsR} is satisfied.
The boundary compatibility condition holds trivially in the absence of traction, as well as in the case of constant normal traction: 

\begin{proposition}
Condition \eqref{eq:integrability_conditions} holds for $\rho=\text{const.}$ and $\tau=0$.
\end{proposition}

\begin{proof}
Let $\omega\in K(\M)$. As Killing fields preserve the volume form, $\delta\omega=0$. Extend  $\rho$ into a constant function $\rho\in C^\infty(\M)$. It follows from the integration by parts formula that
\[
0 = \bra\rho,\delta\omega\ket = \int_{\dM} (\rho,\PntD\omega)_\gD\,\VolumeD,
\]
which is precisely \eqref{eq:integrability_conditions} with $\tau=0$.
\end{proof}

\section{Stress potentials}
\subsection{Existence of stress potentials}

Suppose we are given $\sigma\in\Theta^1(\M)$ satisfying
\[
\deltag\sigma=0.
\]
As explained in the introduction, when $(\M,\g)$ is a simply-connected Euclidean domain, there exists a $\psi\in\Theta^2(\M)$ such that
\[
\sigma=\bHg^*\psi.
\]
A considerable amount of gauge freedom for this $\psi$ suggests itself from this representation; one can alter $\psi$ by any element in $\ker\bHg^*$; another source of freedom is in the boundary conditions. The decomposition \thmref{thm:3.11} provides a complete characterization of the elements in this kernel. 
The following proposition extends this classical representation theorem to manifolds of constant sectional curvature, applicable to arbitrary dimension and topology, supplemented with a choice of gauge, including boundary conditions:

\begin{proposition}
\label{prop:stress_potential}
Let $\sigma\in\Theta^1(\M)$ satisfy $\deltag\sigma=0$. 
Suppose that there exists an $\omega\in\Theta^2(\M)$, satisfying
\beq
\begin{gathered}
\sigma - \bHg^*\omega \perp \calS\BH^1_\NN(\M) 
\Textand
\sigma - \bHg^*\omega \in \ker(\PnnD,\PtnD).
\end{gathered} 
\label{eq:integrability_conditions_stress_potential}
\eeq
Then, there exists a $\psi\in\Theta^2(\M)\cap\ker\frakG$ (i.e., an algebraic curvature) satisfying
\beq
\bHg^*\psi=\sigma 
\qquad 
\psi\in\image \bHg 
\qquad
(\PnnD,\frakT^*)\psi = (\PnnD,\frakT^*)\omega.
\label{eq:stress_potential}
\eeq
If $\sigma$ and $\omega$ are Sobolev sections, then $\psi$ can be chosen to satisfy
\[
\Norm{\psi}_{W^{s+2,p}(\M)} \lesssim 
\Norm{\sigma}_{W^{s,p}(\M)} + \Norm{\PnnD\omega}_{W^{s+2-\frac{1}{p},p}(\dM)}+\Norm{\frakT^*\omega}_{W^{s+1-\frac{1}{p},p}(\dM)}.
\]
\end{proposition}

The proof is basically an application of Theorem.~7.13 in \cite{KL21a}
with $\chi = \sigma$, $\phi = \PnnD\omega$ and $\mu = \frakT^*\omega$. The condition that $\sigma\in\ker(\bHg^*,\bFg)$ follows from \propref{prop:4.2}.

When $(\M,\g)$ is locally-flat and simply-connected, $\calS\BH^1_\NN(\M)=\{0\}$, hence the first condition holds trivially. 
Also, if $\PnnD\sigma=0$ and $\PtnD\sigma=0$, then both conditions hold with the choice of $\omega=0$ provided that $\sigma\bot\calS\BH^1_{\NN}(\M)$. 

We further specialize this representation statement in the physically-relevant case of $d=3$. Then, $\starG\starG^V:\Theta^1(\M)\rightarrow\Theta^2(\M)$ is a $W^{s,p}$-isometry. Using \thmref{thm:5.4} and \propref{prop:stress_potential}, we  provide a sharper choice of gauge, which also facilitates a uniqueness clause.

\begin{theorem}
\label{thm:stress_potential_2_three_dimensions}
Let $d=3$ and $\g$ have constant sectional curvature, and let $\sigma\in\Theta^1(\M)$ satisfy $\sigma\in\image\bHg^*$ (which by \propref{prop:stress_potential} is the case if $\deltag\sigma=0$ and $\sigma$ satisfies the integrability conditions \eqref{eq:integrability_conditions_stress_potential}). Then there exists a $\psi\in\Theta^2(\M)$ solving the set of equations
\beq
\begin{gathered}
\dg\psi=0 
\qquad 
\bHg^*\psi=\sigma
\qquad
(\PttD,\PntD)\psi=0.
\end{gathered}
\label{eq:three_dimensional_potential_2}
\eeq
Such a $\psi$ is unique up to an arbitrary element $\theta\in\calS\BH^2_{\TT}(\M)$. If $\sigma$ is a Sobolev section, then the solution inherits this regularity and for all $1\leq q\leq p$,
\[
\Norm{\psi_\theta}_{W^{s+2,p}(\M)}\lesssim \Norm{\sigma}_{W^{s,p}(\M)}+\Norm{\theta}_{L^{q}(\M)}.
\]
\end{theorem}

Thus, the stress potential $\psi$ can be chosen such to satisfy both algebraic and differential Bianchi identities, along with vanishing boundary conditions. In particular, the boundary conditions for the potential can be chosen independently of the boundary data of $\sigma$.

\begin{proof}
Set $\Sigma=\starG\starG^V\sigma\in\Theta^2(\M)$, and consider the system for $\chi\in\Theta^1(\M)$
\[
\deltag\chi = 0
\qquad  
\bHg\chi=\Sigma
\qquad
(\PtnD,\PnnD)\chi = (0,0).
\]
Since by duality $\sigma\in\image{\bHg^*}$ implies $\Sigma\in\image{\bHg}$, the conditions of \thmref{thm:5.4} are satisfied, hence this system is solvable; the solution $\chi$ is unique up to an arbitrary element in $\calS\BH^1_{\NN}(\M)$. Setting $\psi\in\Theta^2(\M)$,
\[
\psi=\starG\starG^V\chi,
\] 
we obtain by duality that $\psi$ satisfies \eqref{eq:three_dimensional_potential_2}. The uniqueness clause follows from the fact that in $d=3$, $\starG\starG^V:\calS\BH^1_{\NN}(\M)\rightarrow \calS\BH^2_{\TT}(\M)$ is an isometry. Finally, the regularity clause follows from the fact that $\starG\starG^V$ is a $W^{s,p}$-isometry, and the estimate \eqref{eq:sobolev_estimate} holds accordingly. 
\end{proof}


\subsection{Boundary-value problem for the stress potential}

\propref{prop:stress_potential} holds in arbitrary dimension, but the main applications of stress potentials are in dimensions $d=2,3$.
In these cases, one obtains another diagram satisfying the exactness conditions \eqref{eq:exact_relations}:

\[
\begin{tikzcd}
&& {\Theta^3(\M)} \\
&& {} \\
{\Omega^{1,3}(\M)} && {\Theta^2(\M)} && {} \\
&& {} \\
&& {\Theta^1(\M)}
\arrow["{\bHg}"', curve={height=12pt}, from=5-3, to=3-3]
\arrow["{\bHg}"', curve={height=12pt}, from=3-3, to=1-3]
\arrow["{\bHg^*}", curve={height=-12pt}, tail reversed, no head, from=3-3, to=1-3]
\arrow["{\bHg^*}"', curve={height=12pt}, from=3-3, to=5-3]
\arrow["{\tfrac12(\bFg + (\bFg(\cdot))^T)}", curve={height=-12pt}, from=3-1, to=3-3]
\arrow["{\bFg^*}"', curve={height=12pt}, tail reversed, no head, from=3-1, to=3-3]
\end{tikzcd}
\]

In dimension 2, this diagram is trivial since $\Theta^3(\M), \Omega^{3,1}(\M)$ and $\Omega^{1,3}(\M)$ are all $\BRK{0}$. 
In dimension 3, this diagram is a ``reflection" of Diagram \eqref{eq:exact_diagram} via the duality $\psi\mapsto \starG\starG^V\psi$; for example,
\[
\bFg^*:\Theta^2(\M) \to \Omega^{1,3}(\M)
\]
is defined by
\[
\bFg^*|_{\Theta^2(\M)}\psi = \starG\starG^V \bFg|_{\Theta^1(\M)} \starG\starG^V \psi.
\]
For $\psi\in\Theta^1(\M)$, the definition of $\bHg\psi$ coincides with that of Diagram \eqref{eq:exact_diagram}.
Finally, $\Theta^3(\M)$ can be identified with the space $\Omega^0(\M)\simeq C^\infty(\M)$.

In either case, the condition $\psi\in\image{\bHg}$ in \eqref{eq:stress_potential} implies 
\[
\psi\in\ker{(\bFg,\bHg)}.
\]

Let $\sigma$ be a solution of the boundary-value problem \eqref{eq:Euler_lagrange_equations}, and suppose
that the conditions of \propref{prop:stress_potential} hold, namely, there exists an $\omega\in\Theta^2(\M)$ satisfying the required conditions.
Inserting the stress potential $\psi$ into \eqref{eq:Euler_lagrange_equations} one obtains the boundary-value problem
\[
\Bg\psi=\calR 
\qquad 
\psi\in\ker{(\bFg,\bHg)}
\qquad
(\PnnD,\frakT^*)\psi=(\PnnD,\frakT^*)\omega.
\]

For $d=2$ the condition that $\psi\in\ker{(\bFg,\bHg)}$ holds trivially. 
By duality, the boundary-value problem can be reformulated as an equation for a scalar function,
\[
\chi=\starG\starG^V\psi\in C^{\infty}(\M),
\]
namely,
\[
\begin{gathered}
\starG\starG^{V}\bHg^{*}\bHg\psi = \Delta^2_\g\chi  + 2\kappa\Delta_\g \chi+\kappa^2\chi = \starG\starG^V\calR \\
(\PttD,\frakT)\chi=(\chi|_{\dM},\partial_\frakn\chi|_{\dM})=(\star_{\gD}\star_{\gD}^V\omega|_{\dM},\partial_r\star_{\gD}\star_{\gD}^V\omega|_{\dM}),
\end{gathered}
\]
where $\Delta_\g$ is the Laplace-Beltrami operator. In incompatible elasticity, where $(\M,\g)$ is a manifold of constant sectional curvature $\kappa$, \eqref{eq:summary6} with $\calC=-2$ yields $\calR = -2\Rm_\g$, in which case $ \starG\starG^V\calR = -2\kappa$.
This boundary-value problem is a generalization of the bilaplacian equation for the Airy stress functions in incompatible elasticity \cite{MSK14,MSK15}. It is a boundary-value problem for a fourth-order strongly elliptic operator supplemented with Dirichlet boundary conditions, hence has a unique solution.

For $d=3$, the same duality transformation yields a potential $\chi\in\Theta^1(\M)$ satisfying the boundary-value problem,
\[
\begin{gathered}
\Bg\chi= \starG^V\starG\calR \qquad \chi\in\ker{(\bHg^*,\bFg)} \\
(\PttD,\frakT)\chi = (\PttD\star_{\gD}\star_{\gD}^V\omega,\frakT\star_{\gD}\star_{\gD}^V\omega).
\end{gathered}
\]
This system is solvable under the conditions ensuring the existence of a stress potential, however, the solution is generally not unique.

A different choice of gauge for the stress potential under the same assumptions, is the one produced in \thmref{thm:stress_potential_2_three_dimensions}. 
First, we note that for $d=3$,  the dual version of \lemref{lem:boundary_data_delta} reads:

\begin{lemma}
\label{lem:6.3}
Let $d=3$ and let $\psi\in\Theta^2(\M)$, $\rho\in \Omega^2(\dM)$ and $\tau\in\Omega^{1,2}(\dM)$. Suppose that
\[
(\PttD,\PntD)\psi = (\rho,\tau).
\]
Then,
\[
(\frakT,\frakF^*)\psi = (-\dgD \tau,-\deltagD \rho + \tfrac12 \gD\wedge \trace_{\frakh_{0}}\tau)
\]
if and only if
\[
\dg\psi|_{\dM}=0.
\]
\end{lemma}

Inserting $\psi$ into \eqref{eq:Euler_lagrange_equations}, we obtain the system
\[
\Bg\psi=\calR 
\qquad 
\dg\psi=0
\qquad
(\PttD,\frakT,\PntD,\frakF^*)\psi=0,
\]
where  $(\frakT,\frakF^*)\psi=0$ by \lemref{lem:6.3}.
This system is solvable under the conditions ensuring the existence of a stress potential; the solution is again not unique.



\providecommand{\bysame}{\leavevmode\hbox to3em{\hrulefill}\thinspace}
\providecommand{\MR}{\relax\ifhmode\unskip\space\fi MR }
\providecommand{\MRhref}[2]{%
  \href{http://www.ams.org/mathscinet-getitem?mr=#1}{#2}
}
\providecommand{\href}[2]{#2}

\end{document}